\documentclass[a4paper,11pt]{amsart}
\usepackage[utf8x]{inputenc}
\usepackage{amsmath,amssymb,amsthm}
\usepackage{graphics,color,graphicx}
\usepackage{colonequals}
\usepackage{textcomp}

\newtheorem{dfn}{Definition}[section]
\newtheorem{thm}{Theorem}
\newtheorem{thmA}{Theorem}

\newtheorem{thmB}{Theorem}

\newtheorem{thmC}{Theorem}

\newtheorem{prp}[dfn]{Proposition}
\newtheorem{lem}[dfn]{Lemma}
\newtheorem{cor}[dfn]{Corollary}
\newtheorem{quest}{Question}
\newtheorem{prb}{Problem}
\newtheorem{reform}{Reformulation}

\theoremstyle{definition}
\newtheorem{exl}[dfn]{Example}
\newtheorem{rem}[dfn]{Remark}

\newcommand{\R}{{\mathbb R}}
\newcommand{\RP}{{\mathbb R}\mathrm{P}}

\renewcommand{\phi}{\varphi}

\newcommand{\GL}{\operatorname{GL}}

\newcommand{\Proj}{\operatorname{Proj}}
\newcommand{\Aff}{\operatorname{Aff}}
\newcommand{\conv}{\operatorname{conv}}
\newcommand{\inn}{\operatorname{int}}

\newcommand{\grad}{\operatorname{grad}}
\newcommand{\bc}{\gamma}
\newcommand{\vol}{\operatorname{vol}}
\newcommand{\ce}{\colonequals}
\newcommand{\crra}{\operatorname{cr}}
\newcommand{\cone}{\operatorname{cone}}
\newcommand{\const}{\mathrm{const}}
\newcommand{\Ort}{\operatorname{O}}
\newcommand{\dist}{\operatorname{dist}}
\newcommand{\diam}{\operatorname{diam}}
\newcommand{\width}{\operatorname{width}}

\title{Fitting centroids by a projective transformation}
\author{Ivan Izmestiev}
\address{Institut f\"ur Mathematik \\
Freie Universit\"at Berlin \\
Arnimallee 2 \\
D-14195 Berlin, Germany}
\email{izmestiev@math.fu-berlin.de}
\thanks{Supported by the European Research Council under the European Union's Seventh Framework Programme (FP7/2007-2013)/\allowbreak ERC Grant agreement no.~247029-SDModels}

\begin{document}

\begin{abstract}
Given two subsets of $\R^d$, when does there exist a projective transformation that maps them to two sets with a common centroid? When is this transformation unique modulo affine transformations? We study these questions for $0$- and $d$-dimensional sets, obtaining several existence and uniqueness results as well as examples of non-existence or non-uniqueness.

If both sets have dimension $0$, then the problem is related to the analytic center of a polytope and to polarity with respect to an algebraic set. If one set is a single point, and the other is a convex body, then it is equivalent by polar duality to the existence and uniqueness of the Santal\'o point. For a finite point set versus a ball, it generalizes the M\"obius centering of edge-circumscribed convex polytopes and is related to the conformal barycenter of Douady-Earle. If both sets are $d$-dimensional, then we are led to define the Santal\'o point of a pair of convex bodies. We prove that the Santal\'o point of a pair exists and is unique, if one of the bodies is contained within the other and has Hilbert diameter less than a dimension-depending constant. The bound is sharp and is obtained by a box inside a cross-polytope.
\end{abstract}

\maketitle

\section{Introduction}
\subsection{The setup}
This work arose from the following question:
\begin{quote}
Given a convex polytope $P$ in $\R^d$ and a point $q$ inside $P$, does there exist a projective transformation $\phi$ such that $\phi(q)$ is the centroid of the vertices of $\phi(P)$?
\end{quote}
For example, if $P$ is a $d$-simplex, then as $\phi$ one can take the projective transformation that fixes the vertices of $P$ and maps $q$ to the centroid of $P$.

The question can be generalized as follows:
\begin{quest}
\label{quest:1}
Given two subsets $K_1, K_2 \subset \R^d$, does there exist a projective transformation $\phi \colon \RP^d \to \RP^d$ such that the centroids of $\phi(K_1)$ and $\phi(K_2)$ coincide?
\end{quest}



The centroid $\gamma(K)$ can be defined (by the usual integral formula) for any subset $K \subset \R^d$ that has a positive finite $k$-Hausdorff measure for some $0 \le k \le d$, see e.~g.~\cite{KMP06}. Thus, any of the sets $K_i$ in Question \ref{quest:1} may be, say, a convex polytope or the $k$-skeleton of a convex polytope.

If $\dim K = 0$ or $\dim K = d$, then the centroid of $K$ is affinely covariant:
\begin{equation}
\label{eqn:AffCovar}
\bc(\psi(K)) = \psi(\bc(K)) \quad \forall \psi \in \Aff(d)
\end{equation}
so that we can quotient out affine transformations when searching for $\phi$. Since $\dim\Proj(d) - \dim\Aff(d) = d$, the ``number of equations'' becomes equal to the ``number of variables'', and the following question poses itself.

\begin{quest}
\label{quest:2}
If $\dim(K_i) \in \{0, d\}$ in Question \ref{quest:1}, then is $\phi \in \Proj(d)$ such that $\bc(\phi(K_1)) = \bc(\phi(K_2))$ unique up to post-composition with an affine transformation?
\end{quest}

For example, if $K_1$ is the vertex set of a simplex, and $K_2$ is a single point, then $\phi$ is unique in the above sense.

For $\dim(K) \notin \{0, d\}$ the centroid is in general not affinely covariant. Indeed, it is well-known that the centroid of the boundary of a triangle coincides with the centroid of its vertices if and only if the triangle is regular. As the centroid of the vertices is affinely covariant, the centroid of the boundary is not. See \cite{KMP06} for centroids of skeleta of simplices in higher dimensions.


Let us discuss some restrictions we will impose on $K_i$ and $\phi$. First, $K$ is always assumed to be compact and equal to the closure of its interior. The latter is not really a restriction, since replacing $K$ by the closure of $\inn K$ doesn't change the centroid.

Second, we will always assume one set to be contained in the convex hull of the other: $K_2 \subset \conv K_1$. Although this looks quite restrictive, it leaves enough room for non-trivial results. For an idea of what can be done in the case when the convex hulls are incomparable, see Section \ref{sec:FutRes}.

Third, in order for the centroid of $\phi(K)$ to be defined, no point of $K$ may be sent to infinity for $\dim K = 0$ and $\vol(\phi(K) \cap \R^d) < \infty$ must hold for $\dim K = d$. The following restriction (together with compactness of $K$) guarantees both.

\begin{dfn}
\label{dfn:Admissible}
Let $K \subset \R^d \subset \RP^d$. A projective transformation $\phi \colon \RP^d \to \RP^d$ is called \emph{admissible} for $K$, if $\phi(\conv(K)) \subset \R^d$.
\end{dfn}

Non-admissible projective transformations are more difficult to handle; besides, admissible transformations will often suffice. If we allow a projective transformation to send to infinity a hyperplane that separates the points $p_1, \ldots, p_n$, then we can lose the uniqueness, see Proposition \ref{prp:NonAdm}.

The requirement that $K$ is equal to the closure of its interior forbids $K$ to have ``antennas''. It turns out, we heed to forbid ``horns'' in order to ensure the existence of a suitable projective transformation.

%

\begin{dfn}
A $d$-dimensional compact subset $K \subset \R^d$ is called \emph{cusp-free}, if for every $x \in K \cap \partial \conv K$ there is a $d$-simplex contained in $K$ with a vertex at $x$.
\end{dfn}

Examples of cusp-free sets are: pure $d$-dimensional polyhedra (finite unions of convex $d$-dimensional polyhedra); $d$-submanifolds of $\R^d$ with smooth boundary; $d$-submanifolds with corners.

%

\subsection{Making a given point to the centroid of a set}
Here we present our results in the case when $K_2 = \{q\}$ is a single point.

\begin{thm}
\label{thm:OneVsMany}
Let $K = \{p_1, \ldots, p_n\} \subset \R^d$ be a finite set of points affinely spanning $\R^d$, and let $q \in \inn \conv(K)$ be a point in the interior of their convex hull. Then there exists a projective transformation $\phi \colon \RP^d \to \RP^d$, admissible with respect to $K$, such that
\[
\frac{\phi(p_1) + \cdots + \phi(p_n)}n = \phi(q)
\]
If $\psi$ is any other admissible projective transformation with $\bc(\psi(K)) = \psi(q)$, then $\phi \circ \psi^{-1}$ is an affine transformation.
\end{thm}

A projective transformation modulo post-composition with affine ones is uniquely determined by the hyperplane $\ell \subset \R^d$ that it sends to infinity. Associate to every $\ell$ the point $q$ that becomes the centroid of $\{p_i\}$ after $\ell$ is sent to infinity. Then, by Theorem \ref{thm:OneVsMany}, the hyperplanes disjoint from $\conv\{p_1, \ldots, p_n\}$ are in one-to-one correspondence with the points inside the convex hull. This correspondence is related to the polarity with respect to an algebraic set. Namely, let $A$ be the union of the hyperplanes dual to $\{p_i\}$; then the dual of $q$ is the polar with respect to~$A$ of the dual of~$\ell$. See Proposition \ref{prp:Polarity} for more details.

On the other hand, there is a relation to the analytic center of a polytope and the Karmarkar's algorithm, \cite{BL89}.

\begin{thm}
\label{thm:OneVsBody}
Let $K \subset \R^d$ be a compact cusp-free $d$-dimensional set, and $q \in \inn\conv(K)$. Then there exists a projective transformation $\phi \colon \RP^d \to \RP^d$, admissible with respect to $K$, such that
\[
\bc(\phi(K)) = \phi(q)
\]
For any other admissible projective transformation $\psi$ with this property, the composition $\phi \circ \psi^{-1}$ is affine.
\end{thm}
Since projective transformations can be represented by central projections (Section \ref{sec:ConeSec}), Theorem \ref{thm:OneVsBody} can be reformulated as existence and uniqueness of a hyperplane section of a cone through a given point having this point as the centroid. Representing projective transformations as composition of two polarities, we can relate Theorem \ref{thm:OneVsBody} in the case of convex $K$ to the Santal\'o point: the hyperplane that must be sent to infinity is dual to the Santalo point of the dual of $K$. See Theorems \ref{thm:OneVsBodyB}  and \ref{thm:OneVsBodyC} in Section \ref{sec:ConeSecSant}.

If the point $q$ lies sufficiently close to a sufficiently sharp cusp of $K$, then there is no projective transformation making $q$ to the centroid of $K$. See Example \ref{exl:CounterPointBody}.

\begin{figure}[ht]
\begin{center}
\includegraphics{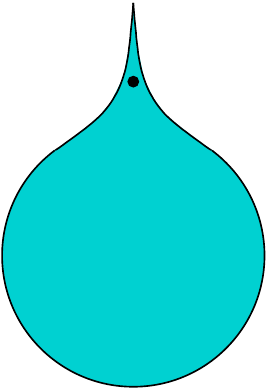}
\end{center}
\caption{A point close to a sharp cusp cannot become centroid.}
\end{figure}

\subsection{One of the sets is finite}
Here we present the results in the case when one of the sets is finite but consists of more than one point.
\begin{thm}
\label{thm:ManyVsMany}
Let $K_1 = \{p_1, \ldots, p_n\} \subset \R^d$ and $K_2 = \{q_1, \ldots, q_m\}$ be such that $K_2 \subset \inn\conv(K_1)$. Then there exists a projective transformation $\phi$, admissible with respect to $K_1$ (and hence with respect to $K_2$), such that $\bc(\phi(K_1)) = \bc(\phi(K_2))$.

In general, $\phi$ is not unique, even up to post-composition with affine transformations.
\end{thm}

For non-uniqueness, see Examples \ref{exl:CounterManyMany1} and \ref{exl:CounterManyMany2}.

\begin{thm}
\label{thm:ManyVsBody}
Let $K \subset \R^d$ be a compact cusp-free $d$-dimensional set, and let $p_1, \ldots, p_n \in \conv K$ be such that every support hyperplane of $K$ contains less than $\frac{n}{d+1}$ of the points $p_1, \ldots, p_n$.
Then there exists a projective transformation $\phi$ such that
\begin{equation}
\label{eqn:PhiBodyPoints}
\bc(\phi(K)) = \frac1{n} \sum_{i=1}^n \phi(p_i)
\end{equation}
In general, $\phi$ is not unique, even modulo affine transformations.
\end{thm}

For the sharpness of the assumptions and for non-uniqueness, see Examples \ref{exl:d+1Points} and \ref{exl:CounterBodyMany}.

Interestingly enough, the assumptions leading to existence become obsolete, and the transformation turns out to be unique, if $K$ is a ball. Since the image of a ball under an admissible projective transformation is an ellipsoid, and the ellipsoid can be mapped back to the ball by an affine transformation, the following theorem is equivalent to the existence and uniqueness of a projective transformation fitting the centroids of a ball and of a finite set.

\begin{thm}
\label{thm:ManyVsBall}
Let $B^d = \{x \in \R^d \mid \|x\| \le 1\}$ be the unit ball centered at the origin, and let $p_1, \ldots, p_n \in B^d$ be a finite set of points, $n \ge 3$. Then there exists a projective transformation fixing $B^d$ such that
\[
\sum_i \phi(p_i) = 0
\]
The transformation $\phi$ is unique up to post-composition with an orthogonal transformation.
\end{thm}

This result generalizes centering via M\"obius transformations \cite{Spr05} used for unique representation of polyhedral types. There is also a relation to the conformal barycenter \cite{DE86}, see Remark \ref{rem:Moebius}.

\subsection{Two convex bodies}
Here we present the results for the case when both $K_1$ and $K_2$ are $d$-dimensional. In order to get some uniqueness results, we need to assume that $K_1$ and $K_2$ are convex. The uniqueness can be guaranteed if one of the bodies lies ``deep inside'' the other.

\begin{dfn}
Let $K_1 \supset K_2$ be two convex bodies in $\R^d$. The \emph{Hilbert diameter} of $K_2$ with respect to $K_1$ is defined as
\[
\diam_{K_1}(K_2) \ce \max_{p,q \in K_2} \frac12 \log |\crra(p,q;a,b)|
\]
where $a,b \in \partial K_1$ are points collinear with $p$ and $q$, and $\crra(p,q;a,b) = \frac{(p-a)(q-b)}{(p-b)(q-a)}$. The \emph{maximum Hilbert width} of $K_2$ with respect to $K_1$ is defined as
\[
\width_{K_1}(K_2) \ce \sup_{\alpha \cap K_1 = \emptyset} \frac12 \log |\crra(m_2, \ell_2; \ell_1, m_1)|
\]
where $\alpha$ is a $(d-2)$-dimensional affine subspace, and $\ell_i, m_i \supset \alpha$ are support hyperplanes to $K_i$.
\end{dfn}
It follows immediately from definition that
\[
\width_{K_1}(K_2) = \diam_{K_2^\circ}(K_1^\circ)
\]
where $K^\circ$ denotes the polar dual of $K$ (one may take the polar duals with respect to any point lying in the interior of both $K_1$ and $K_2$). See Figure \ref{fig:Tang}, where $L_i = K_i^\circ$.

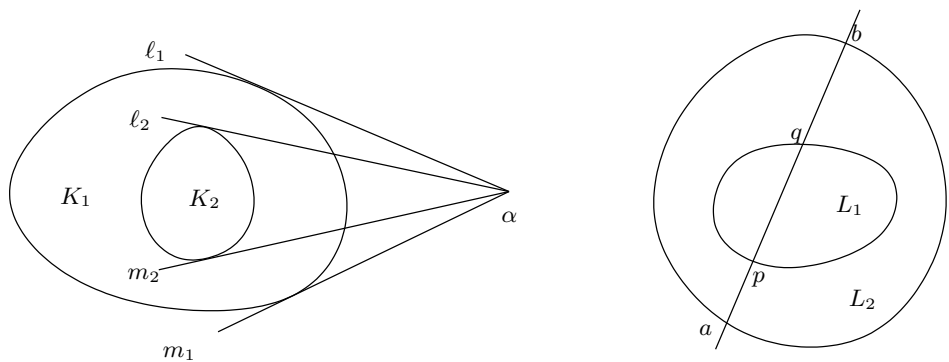
\begin{figure}[ht]
\begin{center}
\begin{picture}(0,0)%
\includegraphics{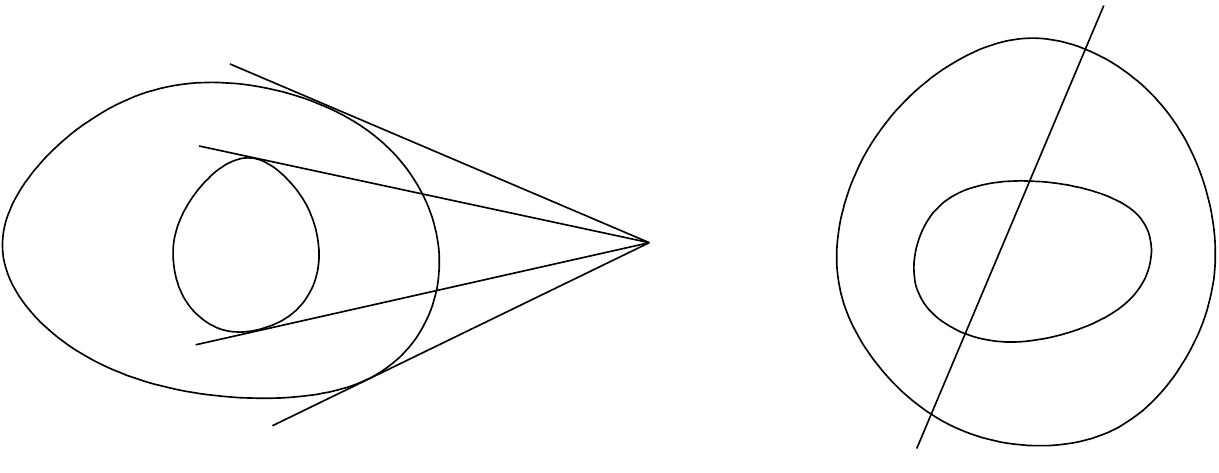}%
\end{picture}%
\setlength{\unitlength}{4144sp}%
\begingroup\makeatletter\ifx\SetFigFont\undefined%
\gdef\SetFigFont#1#2#3#4#5{%
  \reset@font\fontsize{#1}{#2pt}%
  \fontfamily{#3}\fontseries{#4}\fontshape{#5}%
  \selectfont}%
\fi\endgroup%
\begin{picture}(5569,2145)(407,-1110)
\put(1129,325){\makebox(0,0)[lb]{\smash{{\SetFigFont{9}{10.8}{\rmdefault}{\mddefault}{\updefault}{\color[rgb]{0,0,0}$\ell_2$}%
}}}}
\put(1117,-582){\makebox(0,0)[lb]{\smash{{\SetFigFont{9}{10.8}{\rmdefault}{\mddefault}{\updefault}{\color[rgb]{0,0,0}$m_2$}%
}}}}
\put(1221,737){\makebox(0,0)[lb]{\smash{{\SetFigFont{9}{10.8}{\rmdefault}{\mddefault}{\updefault}{\color[rgb]{0,0,0}$\ell_1$}%
}}}}
\put(1326,-1050){\makebox(0,0)[lb]{\smash{{\SetFigFont{9}{10.8}{\rmdefault}{\mddefault}{\updefault}{\color[rgb]{0,0,0}$m_1$}%
}}}}
\put(1478,-144){\makebox(0,0)[lb]{\smash{{\SetFigFont{9}{10.8}{\rmdefault}{\mddefault}{\updefault}{\color[rgb]{0,0,0}$K_2$}%
}}}}
\put(717,-134){\makebox(0,0)[lb]{\smash{{\SetFigFont{9}{10.8}{\rmdefault}{\mddefault}{\updefault}{\color[rgb]{0,0,0}$K_1$}%
}}}}
\put(5310,-185){\makebox(0,0)[lb]{\smash{{\SetFigFont{9}{10.8}{\rmdefault}{\mddefault}{\updefault}{\color[rgb]{0,0,0}$L_1$}%
}}}}
\put(5387,-745){\makebox(0,0)[lb]{\smash{{\SetFigFont{9}{10.8}{\rmdefault}{\mddefault}{\updefault}{\color[rgb]{0,0,0}$L_2$}%
}}}}
\put(3331,-241){\makebox(0,0)[lb]{\smash{{\SetFigFont{9}{10.8}{\rmdefault}{\mddefault}{\updefault}{\color[rgb]{0,0,0}$\alpha$}%
}}}}
\put(4501,-916){\makebox(0,0)[lb]{\smash{{\SetFigFont{9}{10.8}{\rmdefault}{\mddefault}{\updefault}{\color[rgb]{0,0,0}$a$}%
}}}}
\put(4816,-601){\makebox(0,0)[lb]{\smash{{\SetFigFont{9}{10.8}{\rmdefault}{\mddefault}{\updefault}{\color[rgb]{0,0,0}$p$}%
}}}}
\put(5401,839){\makebox(0,0)[lb]{\smash{{\SetFigFont{9}{10.8}{\rmdefault}{\mddefault}{\updefault}{\color[rgb]{0,0,0}$b$}%
}}}}
\put(5041,254){\makebox(0,0)[lb]{\smash{{\SetFigFont{9}{10.8}{\rmdefault}{\mddefault}{\updefault}{\color[rgb]{0,0,0}$q$}%
}}}}
\end{picture}%
\end{center}
\caption{Concurrent tangent hyperplanes to $K_1$ and $K_2$ and the dual collinear points on the boundaries of $L_1$ and $L_2$.}
\label{fig:Tang}
\end{figure}

Note that for $K_1 = B^d$ the number $\frac12 \log |\crra(m_2, \ell_2; \ell_1, m_1)|$ is the hyperbolic distance between $\ell_2$ and $m_2$, with $B^d$ viewed as the Cayley-Klein model of the hyperbolic space. Thus, the maximum hyperbolic width is defined as the maximum distance between support hyperplanes.

\begin{thm}
\label{thm:BodyVsBody}
Let $K_1, K_2 \subset \R^d$ be two convex bodies such that $K_2 \subset \inn\conv(K_1)$. Then there exists a projective transformation $\phi$ such that
\[
\bc(\phi(K_1)) = \bc(\phi(K_2))
\]
In general, $\phi$ is not unique modulo affine transformations. It is unique, if
\begin{equation}
\label{eqn:CR}
\width_{K_1}(K_2) < \log\frac{1 + \kappa_d}{1 - \kappa_d}, \quad \text{where }\kappa_d = \sqrt{\frac{6}{(d+1)(d+2)}}
\end{equation}

The bound \eqref{eqn:CR} is sharp. It is achieved for a cross-polytope $K_1$ and a rectangular parallelepiped $K_2$ inside it, provided that at least one of the sides of the parallelepiped is sufficiently long, see Fig. \ref{fig:SharpBound}.
\end{thm}
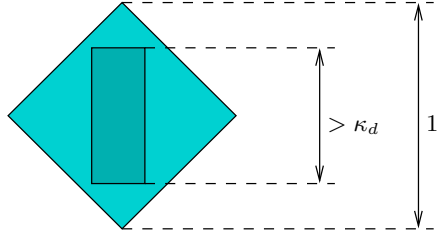
\begin{figure}[ht]
\begin{center}
\begin{picture}(0,0)%
\includegraphics{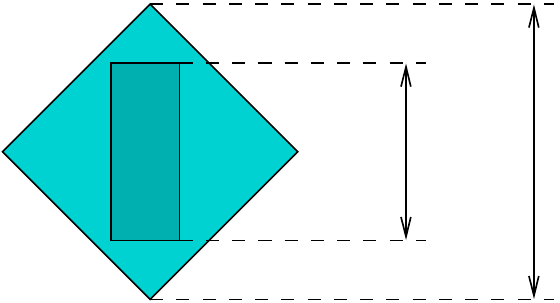}%
\end{picture}%
\setlength{\unitlength}{4144sp}%
\begingroup\makeatletter\ifx\SetFigFont\undefined%
\gdef\SetFigFont#1#2#3#4#5{%
  \reset@font\fontsize{#1}{#2pt}%
  \fontfamily{#3}\fontseries{#4}\fontshape{#5}%
  \selectfont}%
\fi\endgroup%
\begin{picture}(2544,1374)(664,-523)
\put(2566, 74){\makebox(0,0)[lb]{\smash{{\SetFigFont{9}{10.8}{\rmdefault}{\mddefault}{\updefault}{\color[rgb]{0,0,0}$>\kappa_d$}%
}}}}
\put(3151, 74){\makebox(0,0)[lb]{\smash{{\SetFigFont{9}{10.8}{\rmdefault}{\mddefault}{\updefault}{\color[rgb]{0,0,0}$1$}%
}}}}
\end{picture}%
\end{center}
\caption{For a long box inside a cross-polytope a projective transformation is not unique.}
\label{fig:SharpBound}
\end{figure}

The discussion in Section \ref{sec:CompPolar} justifies the following definition.

\begin{dfn}
A point $y \in \R^d$ is called a \emph{Santal\'o point of a pair} $L_1, L_2 \subset \R^d$ of convex bodies, if $(L_1)^\circ_y$ and $(L_2)^\circ_y$ have the same centroid. Here $L^\circ_y$ denotes the polar dual of $L$ with respect to the unit sphere centered at $y$.
\end{dfn}

By going to the polar duals of $K_1$ and $K_2$, one derives from Theorem \ref{thm:BodyVsBody} criteria for existence and uniqueness of a Santal\'o point of a pair.

\begin{cor}
Let $L_1, L_2 \subset \R^d$ be two convex bodies such that $L_2 \subset \inn L_1$. Then the pair $(L_1, L_2)$ has at least one Santal'\o point.

If the Hilbert diameter of $L_2$ with respect to $L_1$ satisfies
\[
\diam(L_2) < \log \frac{1+\kappa_d}{1-\kappa_d}
\]
then the Santal\'o point of the pair $(L_1, L_2)$ is unique.
\end{cor}

\begin{thm}
\label{thm:BodyVsBall}
Let $K \subset B^d$ be a convex body such that
\begin{equation}
\label{eqn:BallBody1}
\width(K) < \log\frac{1+\sqrt{\frac3{d+2}}}{1-\sqrt{\frac3{d+2}}}
\end{equation}
in the hyperbolic metric defined by $B^d$ as a Cayley-Klein model.
Then there is a unique projective transformation $\phi$, up to post-composition with orthogonal ones, that fixes $B^d$ and such that the centroid of $\phi(K)$ is the center of~$B^d$. The bound in \eqref{eqn:BallBody1} is sharp.
\end{thm}

\begin{thm}
\label{thm:BallVsBody}
Let $K \subset B^d$ be a convex body such that
\begin{equation}
\label{eqn:BallBody2}
\diam(K) < \log\frac{1+\sqrt{\frac2{d+1}}}{1-\sqrt{\frac2{d+1}}}
\end{equation}
in the hyperbolic metric defined by $B^d$ as a Cayley-Klein model.
Then there exists a unique point $y \in \inn K$ such that the centroids of the polars of $K$ and $B^d$ with respect to $y$ coincide. The bound in \eqref{eqn:BallBody2} is sharp.
\end{thm}

\begin{cor}
Let $B^d$ be the unit ball, and $K$ be contained in a concentric ball of radius $\sqrt{\frac3{d+2}}$. Then there is a unique projective transformation $\phi$, up to post-composition with orthogonal ones, that fixes $B^d$ and such that the centroid of $\phi(K)$ is the center of $B^d$.

Let $B^d$ be the unit ball, and $K$ be contained in a concentric ball of radius $\sqrt{\frac2{d+1}}$.
Then there is a unique Santal\'o point of the pair $(B^d, K)$.
\end{cor}

\subsection{Plan of the paper and acknowledgments}
In Section \ref{sec:Reform} we discuss left cosets of the affine group in the projective group and represent them by elations, central projections and compositions of polarities.

Section \ref{sec:FinSets} deals with the case of $\dim K_1 = \dim K_2 = 0$, that is with finite point sets. Theorems \ref{thm:OneVsMany} and \ref{thm:ManyVsMany} are proved here. The solution is found as a critical point of a concave functional \eqref{eqn:F1}, respectively of the difference of two such functionals.

In Section \ref{sec:PointBody}, Theorem \ref{thm:OneVsBody} is proved and interpreted in the contexts of minimizing the volume of a cone section and of the Santal\'o point. Here the convex functional \eqref{eqn:FBody} associated with a convex body is introduced.

Section \ref{sec:ManyVsBody} deals with the case $\dim K_1 = d$, $\dim K_2 = 0$ and proves Theorems \ref{thm:ManyVsBody} and \ref{thm:ManyVsBall}.

Section \ref{sec:TwoBodies} deals with the case of two $d$-dimensional sets and proves Theorems \ref{thm:BodyVsBody}.

Finally, in Section \ref{sec:FutRes} poses some questions for future research.

The author wishes to thank Arnau Padrol, Raman Sanyal, Boris Springborn, and G\"unter Ziegler for useful discussions.

\section{Three ways to represent projectivities modulo affinities}
\label{sec:Reform}
\subsection{Choosing a hyperplane to be sent to infinity}
\label{sec:HypInfty}
Affine transformations (or affinities) of $\R^d$ are maps of the form $x \mapsto Ax + b$, where $A \in \GL(d)$.
Identify $\R^d$ with a subset of the projective space:
\[
\R^d = \{(x_0 : x_1 : \ldots : x_d) \in \RP^d \mid x_0 \ne 0\}
\]
by associating $x \in \R^d$ with the equivalence class of $(1,x) \in \R^{d+1}$.
Then projective transformations of $\RP^d$ restricted to $\R^d$ have the form
\[
x \mapsto \frac{Ax+b}{\langle c, x \rangle + \delta}, \quad A \in \GL(d), b, c \in \R^d, \delta \in \R
\]
In particular, the group $\Aff(d)$ of affinities is a subgroup of the group $\Proj(d)$ of projectivities.

\begin{prp}
\label{prp:A}
Every right coset of $\Aff(d)$ in $\Proj(d)$ has a unique representative of the form
\begin{equation}
\label{eqn:PhiY}
\phi_y \colon x \mapsto \frac{x}{1 + \langle x, y \rangle}
\end{equation}
or two representatives ($y$ doing the same as $-y$) of the form
\begin{equation}
\label{eqn:NotElation}
x \mapsto \frac{x + y}{\langle x, y \rangle}, \quad \|y\|=1
\end{equation}
\end{prp}
\begin{proof}
Two projectivities belong to the same right coset of $\Aff(d)$ if and only if they send to infinity the same hyperplane. Any hyperplane that does not pass through the origin has equation $\langle x, y \rangle + 1 = 0$ for a unique $y$, and is therefore sent to infinity by a map of the form \eqref{eqn:PhiY}. In particular, for $y=0$ the hyperplane at infinity is sent to itself. Any hyperplane through the origin is sent to infinity by a map of the form \eqref{eqn:NotElation}.
\end{proof}

\begin{rem}
Projective transformation \eqref{eqn:PhiY} is an elation with the axis $y^\perp$ and center $0$. Projective transformation \eqref{eqn:NotElation} is not an elation (not even a homology), but can be replaced by the map $x \mapsto \frac{x+y}{\langle x, y \rangle} - y$ which, for $\|y\| = 1$, is an elation with the axis $\{x \mid \langle x, y \rangle + 1 = 0\}$ and center $-y$.
\end{rem}

We may always assume $0 \in \conv(K_i)$ for $i=1,2$. Then none of the maps \eqref{eqn:NotElation} is admissible in the sense of Definition \ref{dfn:Admissible}, and the map \eqref{eqn:PhiY} is admissible if and only if $y \in \inn K_1^\circ \cap \inn K_2^\circ$, where
\begin{equation}
\label{eqn:PolarBody}
K^\circ \ce \{y \in \R^d \mid \langle x, y \rangle \ge -1\, \forall x \in K\}
\end{equation}
This allows us to reformulate Questions \ref{quest:1} and \ref{quest:2} as follows.

\begin{reform}
\label{ref:A}
For a set $K \subset \R^d$ containing the origin in the interior of the convex hull, when does there exist $y \in \R^d$ such that $\bc(\phi_y(K)) = 0$?

For two sets $K_1, K_2 \subset \R^d$ containing the origin in their convex hulls, when does there exist $y \in \inn K_1^\circ \cap \inn K_2^\circ$ such that $\bc(\phi_y(K_1)) = \bc(\phi_y(K_2))$?

Under what assumptions is $y$ unique?
\end{reform}

\subsection{Cone sections}
\label{sec:ConeSec}
For every set $X \subset \R^{d+1}$ define the conical hull over $X$ as
\begin{equation}
\label{eqn:ConeHull}
\cone(X) = \{\lambda x \mid x \in X\}
\end{equation}
For $x \in \R^d$ denote
\[
\hat x \ce (1,x) \in \R^{d+1}
\]
and for $K \subset \R^d$ denote $\hat K \ce \{\hat x \mid x \in K\}$. This associates with $K \subset \R^d$ a set $C = \cone(\hat K)$ in $\R^{d+1}$. Hyperplane sections of $C$ are central projections of $\hat K$, and thus images of $K$ under projective transformations. This leads to the following reformulation of Questions \ref{quest:1} and \ref{quest:2}.

\begin{reform}
\label{ref:B}
For a cone $C \subset \R^{d+1}$ and a point $q$ different from the apex of $C$, when does there exist a hyperplane $H$ through $q$ such that $q$ is the centroid of $H \cap C$?

For two cones $C_1, C_2 \subset \R^{d+1}$ with common apex, when does there exist a hyperplane $H$ not passing through the apex such that the centroids of $H \cap C_1$ and $H \cap C_2$ coincide?

Under what assumptions is $H$ unique (in the second case, up to parallel translation)?
\end{reform}

To show that central projections modulo dilations correspond to projectivities modulo affinities, let us relate central projections with transformations $\phi_y$ from \eqref{eqn:PhiY}. For every vector $y \in \R^d$ denote
\begin{equation}
\label{eqn:HY}
H_y \ce \{(x_0, x) \in \R^{d+1} \mid \langle x, y \rangle + x_0 - 1 = 0\}
\end{equation}
Denote by $\rho_y \colon \widehat{\R^d} \to H_y$ the central projection and by $\pi \colon H_y \to \widehat{\R^d}$ the parallel projection along $e_0$.

\begin{prp}
\label{prp:B}
The map \eqref{eqn:PhiY} is a composition of a central and a parallel projection:
\[
\phi_y = \pi \circ \rho_y
\]
\end{prp}
\begin{proof}
From $\rho_y(\hat x) = \lambda \hat x$ and $\rho_y(\hat x) \in H_y$ it follows that
\[
\rho_y(\hat x) = \frac{\hat x}{1 + \langle x, y \rangle}
\]
And since $\pi(\hat x) = x$, we have $\phi_y = \pi \circ \rho_y$.
\end{proof}

Note that the hyperplane \eqref{eqn:HY} contains $e_0$, and that every non-vertical hyperplane through $e_0$ is $H_y$ for some $y \in \R^d$. This establishes a bijection between the images $\phi_y(K)$ and sections $H_y \cap C$ and shows that Reformulation \ref{ref:B} is equivalent to Reformulation \ref{ref:A}.

\subsection{Composition of two polarities}
\label{sec:CompPolar}
Similarly to \eqref{eqn:PolarBody}, define the polar dual of $L$ with respect to a point $y$:
\[
L_y^\circ \ce (L-y)^\circ + y = \{x \in \R^d \mid \langle x-y, z-y \rangle \ge -1 \; \forall z \in L\}
\]
Here is a reformulation of Questions \ref{quest:1} and \ref{quest:2} in the case when $K_1$ and $K_2$ are both convex bodies.

\begin{reform}
\label{ref:C}
For a convex body $L \subset \R^d$, when does there exist a point $y \in \inn L$ such that the polar dual of $L$ with respect to $y$ has centroid at $y$?

For two convex bodies $L_1, L_2 \subset \R^d$, when does there exist a point $y \in \inn L_1 \cap \inn L_2$ such that the centroids of the polar duals of $L_1$ and $L_2$ with respect to $y$ coincide?

Under what assumptions is $y$ unique?
\end{reform}

Again, we justify this by relating polarity with variable center to the map $\phi_y$ from \eqref{eqn:PhiY}.
\begin{prp}
\label{prp:C}
For every $d$-dimensional convex body $K \subset \R^d$ and every point $y \in \inn K^\circ$ we have
\[
\phi_y(K) = (K^\circ - y)^\circ = (K^\circ)^\circ_y - y
\]
\end{prp}
\begin{proof}
Indeed, for any $y \in \inn K^\circ$ we have
\begin{multline*}
z \in K^\circ \Leftrightarrow \langle x, z \rangle \ge -1 \; \forall x \in K \Leftrightarrow \langle x, z-y \rangle \ge -(1 + \langle x, y \rangle) \; \forall x \in K \\
\Leftrightarrow \left\langle \frac{x}{1 + \langle x, y \rangle}, z-y \right\rangle \ge -1 \; \forall x \in K \Leftrightarrow z - y \in (\phi_y(K))^\circ
\end{multline*}
Hence $K^\circ - y = (\phi_y(K))^\circ$ for every $K \subset \R^d$. If $K$ is convex, compact, and $0 \in \inn K$, then $(\phi_y(K))^{\circ\circ} = \phi_y(K)$, and the proposition follows.
\end{proof}

\begin{rem}
The property $y = \bc(L^\circ_y)$ is characteristic for the Santal\'{o} point of $L$, \cite[Remark 10.8]{Leicht98}. Thus, existence and uniqueness of the Santal\'{o} point for convex bodies implies a positive answer to Questions \ref{quest:1} and \ref{quest:2} in the case of a convex body and a point.

In the case of two convex bodies $L_1$ and $L_2$ in Reformulation \ref{ref:C} the point $y$ can be called the Santal\'{o} point of a pair of convex bodies.
\end{rem}

\section{Fitting centroids of two finite sets}
\label{sec:FinSets}
\subsection{One point vs. several}
Here we prove Theorem \ref{thm:OneVsMany} using Reformulation \ref{ref:A} from Section \ref{sec:HypInfty}. Without loss of generality we may assume $q=0$. Since $\phi_y(0) = 0$ for all $y$, it follows from Proposition \ref{prp:A} that Theorem \ref{thm:OneVsMany} is equivalent to the following.

\begin{thmA}
\label{thm:OneVsManyA}
Let $p_1, \ldots, p_n \in \R^d$ be such that $0 \in \inn P$, where $P = \conv\{p_1, \ldots, p_n\}$. Then there exists a unique $y \in \inn P^\circ$ such that
\begin{equation}
\label{eqn:Cond1}
\sum_{i=1}^n \frac{p_i}{1 + \langle p_i, y \rangle} = 0
\end{equation}
\end{thmA}

The proof is based on the fact that the left hand side of \eqref{eqn:Cond1} is the gradient of a strictly concave function. Define
\begin{equation}
\label{eqn:F1}
F \colon \inn P^\circ \to \R, \quad F(y) \ce \sum_{i=1}^n \log(1 + \langle p_i, y \rangle)
\end{equation}

\begin{lem}
\label{lem:Grad1}
We have
\[
\grad F(y) = \sum_{i=1}^n \frac{p_i}{1 + \langle p_i, y \rangle}
\]
\end{lem}
\begin{proof}
Indeed, for every $x \in \R^d$ and every $i$ we have
\[
D_u(\log(1 + \langle p_i, y \rangle)) = \frac{D_u(1 + \langle p_i, y \rangle)}{1 + \langle p_i, y \rangle} = \frac{\langle p_i, u \rangle}{1 + \langle p_i, y \rangle}
\]
\end{proof}

\begin{lem}
\label{lem:Conv1}
The function $F$ is strictly concave.
\end{lem}
Basically, this follows from the strict concavity of $\log x$ on $\R$, as $F$ is a sum of logarithms of affine functions whose linear parts span $(\R^d)^*$.
\begin{proof}
Computing the second derivative of $F$ yields
\[
D_{u,u}^2F(y) = -\sum_i \frac{\langle p_i, u \rangle^2}{(1 + \langle p_i, y \rangle)^2} \le 0
\]
Besides, $D_{u,u}^2F(y) = 0$ if and only if $\langle p_i, u \rangle = 0$ for all $i$. As $p_i$ are affinely spanning $\R^d$, they are also linearly spanning it, so that all scalar products vanish only if $u = 0$.
\end{proof}

\begin{lem}
\label{lem:Coerc1}
The value $F(y)$ tends to $-\infty$ as $y$ tends to $\partial P^\circ$.
\end{lem}
\begin{proof}
As $y \to \partial P^\circ$, some of $1 + \langle p_i, y \rangle$ tend to $0$, and their logarithms tend to $-\infty$. On the other hand, since $P^\circ$ is bounded, all summands in \eqref{eqn:F1} are bounded from above. Thus the whole sum tends to $-\infty$ as $y \to \partial P^\circ$.
\end{proof}

\begin{proof}[Proof of Theorem \ref{thm:OneVsManyA}]
By Lemma \ref{lem:Grad1}, a point $y \in \inn P^\circ$ satisfies \eqref{eqn:Cond1} if and only if $y$ is a critical point of the function $F$ from \eqref{eqn:F1}. By Lemma \ref{lem:Coerc1}, $F$~attains a maximum on $\inn P^\circ$. The point of maximum is a critical point, and this proves the existence part of the theorem.

To prove the uniqueness, use the strict concavity of $F$, Lemma \ref{lem:Conv1}. It implies that all critical points of $F$ are strict local maxima. Since $\inn P^\circ$ is convex, there cannot be more than one strict local maximum, and the theorem is proved.
\end{proof}

\subsection{The same, from a homogeneous point of view}
\label{sec:HomogPoints}
Here we repeat the argument from the previous section in the spirit of Section \ref{sec:ConeSec}. This serves as a preparation to some of the arguments that will follow.

Let $R_1, \ldots, R_n \subset \R^{d+1}$ be open rays issued from the origin, linearly spanning $\R^{d+1}$, and contained in an open half-space whose boundary goes through the origin. Denote $C = \conv(\bigcup_{i=1}^n R_i)$. We want to show that for every open ray $S$ issued from the origin and contained in the interior of $C$ there is a hyperplane $H$ that intersects all rays and whose intersection point with $S$ is the centroid of the intersection points with $\{R_i\}$:
\begin{equation}
\label{eqn:BC1b}
H \cap S = \bc(H \cap R_1, \ldots, H \cap R_n)
\end{equation}
For this we choose arbitrary points $p_i \in R_i$, $q \in S$ and introduce the function
\begin{equation}
\label{eqn:F1Homog}
F(y) \ce \frac1n \sum_{i=1}^n \log \langle p_i, y \rangle - \log \langle q, y \rangle
\end{equation}
defined in the interior of the cone
\[
C^* \ce \{y \in \R^{d+1} \mid \langle p_i, y \rangle \ge 0\}
\]
(since all $p_i$ lie in an open half-space, $\inn C^* \ne \emptyset$). We compute the gradient
\[
\grad F(y) = \frac1n \sum_{i=1}^n \frac{p_i}{\langle p_i, y \rangle} - \frac{q}{\langle q, y \rangle}
\]
and find that $\grad F(y) = 0$ if and only if the hyperplane $\{x \mid \langle x, y \rangle = 1\}$ satisfies the condition \eqref{eqn:BC1b}.
On the other hand,
\[
D^2_{u,u}F(y) = -\frac1n \sum_{i=1}^n \frac{\langle p_i, u \rangle^2}{\langle p_i, y \rangle^2} + \frac{\langle q, u \rangle^2}{\langle q, y \rangle^2}
\]
As we will see in a minute, the function $F$ is neither convex nor concave, which complicates the search for a critical point. The remedy is to restrict $F$ to the hyperplane $\{y \mid \langle q, y \rangle = 1\}$.

\begin{lem}
\label{lem:CritPtProj}
Critical points of $F$ restricted to $\{y \mid \langle q, y \rangle = 1\}$ correspond to hyperplanes $H$ through $q$ that satisfy the condition \eqref{eqn:BC1b}.
\end{lem}
\begin{proof}
Indeed, $y$ is a critical point of the restriction if and only if $\grad F(y)$ is orthogonal to $\{y \mid \langle q, y \rangle = 1\}$, that is collinear with $q$.
\end{proof}

\begin{lem}
\label{lem:ConcProj}
The restriction of the function $F$ to $\inn C^* \cap \{y \mid \langle q, y \rangle = 1\}$ is strictly concave.
\end{lem}
\begin{proof}
A vector $u$ is tangent to the hyperplane $\{y \mid \langle q, y \rangle = 1\}$ if and only if $\langle q, u \rangle = 0$. Since $\{p_i\}$ linearly span $\R^{d+1}$, this implies $D^2_{u,u} > 0$ for $u \ne 0$.
\end{proof}

Lemmas \ref{lem:CritPtProj} and \ref{lem:ConcProj} imply the existence and uniqueness of a hyperplane $H$ through $q$ for which $q$ is the centroid of the intersection points with the rays $\{R_i\}$. This gives another proof of Theorem \ref{thm:OneVsMany}, now in Formulation~\ref{ref:B}.

\begin{prp}
The second derivative of the function $F$ has signature $(0, -, \ldots, -)$ at the critical points of $F$ and signature $(+, -, \ldots, -)$ at the non-critical points.
\end{prp}
\begin{proof}
For any $y \in \inn C^*$ the vector $y$ is isotropic for $D^2F(y)$:
\[
D^2_{y,y}F(y) = 0
\]
Since by Lemma \ref{lem:ConcProj} the quadratic form $D^2F(y)$ has a $d$-dimensional positive subspace, its signature is either $(0, -, \ldots, -)$ or $(+, -, \ldots, -)$, depending on whether the isotropic vector $y$ belongs to the kernel or not. It belongs to the kernel if and only if $\grad F(y) = 0$.
\end{proof}

\subsection{Polarity with respect to a union of hyperplanes}
Fix the points $p_1, \ldots, p_n$ affinely spanning $\R^d$. Theorem \ref{thm:OneVsMany} says that there is a bijection between the points in the interior of $\conv\{p_i\}$ and the hyperplanes disjoint from $\conv\{p_i\}$: a hyperplane $\ell$ corresponds to the point $q \in \inn \conv\{p_i\}$ that becomes the centroid of $\{p_i\}$ when $\ell$ is sent to infinity by a projective transformation. (In particular, the hyperplane at infinity corresponds to the actual centroid of $\{p_i\}$.) In this section we will relate this correspondence to the polarity with respect to an algebraic set.

For a homogeneous degree $n$ polynomial $f$ on a vector space $V$ denote by the same letter $f$ the corresponding $n$-linear symmetric form:
\[
f(x_1, \ldots, x_n) = f(x_{\sigma(1)}, \ldots, x_{\sigma(n)}), \quad f(x, \ldots, x) = f(x)
\]
Let $P(V) = V/x \sim \lambda x$ denote the projectivization of the vector space $V$.
\begin{dfn}
The \emph{$(n-1)$-st kernel} of the $n$-linear symmetric form $f$ is
\[
\ker_{n-1}f \ce \{x \in V \mid f(x, \ldots, x, y) = 0 \, \forall y \in V\}
\]
The \emph{$(n-1)$-st polar} of a point $[x] \in P(V \setminus \ker_{n-1}f)$ with respect to the projective algebraic set $\{[x] \in V \mid f(x) = 0\}$ is the projective hyperplane
\[
\{[y] \in V \mid P(x, \ldots, x, y) = 0\}
\]
\end{dfn}
By the canonical duality, hyperplanes in $V$ correspond to one-dimensional subspaces of $V^*$. Thus the polarity determines a map $P(V \setminus \ker_{n-1}f) \to P(V^*)$. Below we consider a polynomial $f$ on $V^*$, therefore will have to do with a map
\begin{equation}
\label{eqn:Polarity}
P(V^* \setminus \ker_{n-1}f) \to P(V)
\end{equation}

\begin{prp}
\label{prp:Polarity}
Let $\ell_1, \ldots, \ell_n$, and $m$ be $1$-dimensional subspaces of a vector space $V$, and let $\rho \in V^*$ be a linear functional on $V$ such that $\ker \rho$ doesn't contain any of $\ell_i, m$. Then the following conditions are equivalent.
\begin{enumerate}
\item
The centroid of the points where the lines $\ell_1, \ldots, \ell_n$ intersect the hyperplane $\{\rho(X) = 1\}$ lies on the line $m$.
\item
The polar of $[\rho] \in P(V^*)$ with respect to the algebraic set $\bigcup_{i=1}^n \ker(\ell_i) \subset P(V^*)$ is $m \in P(V)$.
\end{enumerate}
\end{prp}
\begin{proof}
Choose arbitrary points $p_i \in \ell_i$ and $q \in m$ different from the origin. By assumption, $\rho(p_i) \ne 0$, $\rho(q) \ne 0$. The intersection points of the corresponding lines with $\{\rho(X) = 1\}$ are
\begin{equation}
\label{eqn:PointsOnLines}
\frac{p_1}{\rho(p_1)}, \ldots, \frac{p_n}{\rho(p_n)}, \text{ and } \frac{q}{\rho(q)}
\end{equation}
Thus the first condition is equivalent to
\begin{equation}
\label{eqn:FirstCond}
\frac{q}{\rho(q)} = \frac1n \sum_{i=1}^n \frac{p_i}{\rho(p_i)}
\end{equation}

On the other hand, for the polynomial
\[
f(\rho) = \rho(p_1) \cdot \ldots \cdot \rho(p_n)
\]
on $V^*$ we have
\[
f(\rho, \ldots, \rho, \sigma) = \frac1n f(\rho) \sum_{i=1}^n \frac{\sigma(p_i)}{\rho(p_i)}
\]
Thus the polar of $\rho$ is the following hyperplane in $V^*$:
\[
\left\{\sigma \in V^* \,\left|\, \sum_{i=1}^n \frac{\sigma(p_i)}{\rho(p_i)} = 0\right\}\right.
\]
Dually, this is the $1$-dimensional subspace of $V$ spanned by $\sum_{i=1}^n \frac{p_i}{\rho(p_i)}$. Thus the second condition is equivalent to
\[
q = \lambda \sum_{i=1}^n \frac{p_i}{\rho(p_i)}
\]
which is equivalent to \eqref{eqn:FirstCond} and thus to the first condition.
\end{proof}

In the Reformulation B of our problem about projective transformations (see Section \ref{sec:HomogPoints}) we intersect with a hyperplane not a collection of lines, but a collection of rays. This means that each of the points \eqref{eqn:PointsOnLines} is assumed to lie in a specified half of the corresponding line, i.~e. the number $\rho(p_i)$, respectively $\rho(q)$ must have a specified sign. In other words, the point $[\rho] \in P(V^*)$ must lie in a specified component of the complement $P(V^*) \setminus \bigcup_{i=1}^n \ker p_i$. By counting the components one can determine the multiplicity of the map \eqref{eqn:Polarity}, i.~e. the number of classes of projective transformations that send $q$ to the centroid of the images of $\{p_i\}$.

\begin{prp}
\label{prp:NonAdm}
Let $p_1, \ldots, p_n, q \in \R^d$ be in general position, that is each $d+2$ of them affinely independent. Then there are exactly $\frac{(n-1)(n-2)}2$ equivalence classes of projective transformations $\phi$ modulo post-composition with affine transformations such that
\[
\frac1n \sum_{i=1}^n \phi(p_i) = \phi(q)
\]
\end{prp}
\begin{proof}
Let $\hat{p_i} = (1,p_i) \in \R^{d+1}$ and $\hat q = (1,q) \in \R^{d+1}$. Let $\ell_i, m \subset \R^{d+1}$ be the $1$-dimensional subspace spanned by $\hat{p_i}$, respectively by $\hat q$. Let $\rho \in (\R^{d+1})^*$ be such that $\rho(\hat{p_i}) \ne 0$, $\rho(\hat q) \ne 0$.

For the points \eqref{eqn:PointsOnLines} (with $\hat{p_i}$ and $\hat q$ instead of $p_i$ and $q$) in the affine hyperplane $\{\rho(x) = 1\}$ there is a unique class of admissible projective transformations making $q$ to the centroid of $\{p_i\}$ if and only if $\frac{\hat q}{\rho(\hat q)}$ lies in the interior of the convex hull of $\frac{\hat{p_i}}{\rho(\hat{p_i})}$. The latter condition says that $\rho$ belongs to a component of $P(V^*) \setminus \bigcup_{i=1}^n \ker \hat{p_i}$ that is disjoint from $\ker \hat q$. Besides, any two functionals from the same component give rise to the same class of projective transformations. There are $\frac{n(n-1)}2 + 1$ components in total, and $\ker \hat q$ intersects $n$ of them, which leads to the number in the proposition.
\end{proof}

%
%

\subsection{Several points vs. several}
\begin{proof}[Proof of Theorem \ref{thm:ManyVsMany}]
Without loss of generality $0 \in \conv(K_1)$, so that the hyperplane sent to infinity by $\phi$ cannot pass through the origin. By Proposition \ref{prp:A} we may look for $\phi$ among the maps of the form \eqref{eqn:PhiY}. The condition $\bc(\phi(K_1)) = \bc(\phi(K_2))$ then says
\begin{equation}
\label{eqn:ManyMany}
\frac1n \sum_{i=1}^n \frac{p_i}{1 + \langle p_i, y \rangle} = \frac1m \sum_{j=1}^m \frac{q_j}{1 + \langle q_j, y \rangle}
\end{equation}

Similar to the proof of Theorem \ref{thm:OneVsManyA}, the solutions of \eqref{eqn:ManyMany} are the critical points of the function
\begin{equation}
\label{eqn:ManyManyGrad}
\frac1n \sum_{i=1}^n \log(1 + \langle p_i, y \rangle) - \frac1m \sum_{j=1}^m \log(1 + \langle q_j, y \rangle)
\end{equation}
defined in the interior of $(\conv K_1)^\circ$.
The assumption $\conv K_2 \subset \inn \conv K_1$ implies $(\conv K_1)^\circ \subset \inn ((\conv K_2)^\circ)$, so that the function tends to $-\infty$ as $y$ tends to the boundary of $\conv(K_1)^\circ$, and hence attains its maximum. The point of minimum yields a desired projective transformation.

For non-uniqueness, see Examples \ref{exl:CounterManyMany1} and \ref{exl:CounterManyMany2}.
\end{proof}

\begin{rem}
The sum \eqref{eqn:ManyManyGrad} may tend to $-\infty$ under less restrictive assumptions than $\conv K_2 \subset \inn \conv K_1$. For example, it does so when $K_1$ consists of the vertices of a triangle in $\R^2$, and $K_2$ of three points on the sides of the triangle.

On the other hand, if $K_1$ is the vertex set of a tetrahedron in $\R^3$, and $K_2$ consists of three points on one edge and two points on the opposite edge, then there is no projective transformation that fits the centroids of $K_1$ and $K_2$ (the centroid of a tetrahedron lies in the plane parallel to a pair of opposite edges and equidistant from them). In particular, in this case the sum \eqref{eqn:ManyManyGrad} does not tend to $-\infty$ near the boundary of the domain.
\end{rem}

\begin{figure}[ht]
\begin{center}
\includegraphics{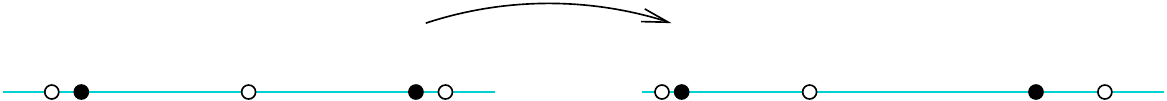}
\end{center}
\caption{Example \ref{exl:CounterManyMany1}: two subsets of the line whose centroids can be fitted together in different ways.}
\end{figure}

\begin{exl}
\label{exl:CounterManyMany1}
Let $d=1$, $K_1 = \{-1, 0, 1\}$, $K_2 = \left\{-\frac{3}{\sqrt{13}}, \frac{3}{\sqrt{13}}\right\}$. Then $\bc(K_1) = \bc(K_2)$, so that one solution is the identity map. There is another solution $\phi(x) = \frac{x}{3-x}$. Indeed, we have
\[
\phi(K_1) = \left\{ -\frac14, 0, \frac12 \right\}, \quad \phi(K_2) = \left\{ -\frac1{\sqrt{13}+1}, \frac1{\sqrt{13}-1} \right\}
\]
and thus $\bc(\phi(K_1)) = \frac1{12} = \bc(\phi(K_2))$.
\end{exl}

One may argue that the above example only works because $K_1$ is not in convex position. For $d=1$ this is actually true: if $K_1 = \{-1, 1\}$ and $K_2 \subset (-1, 1)$, then the solution is unique. In higher dimensions this does not help, as the following example shows. The reason for the failure is that even if $K_1$ is in convex position, its projections are not.

\begin{figure}[ht]
\begin{center}
\includegraphics{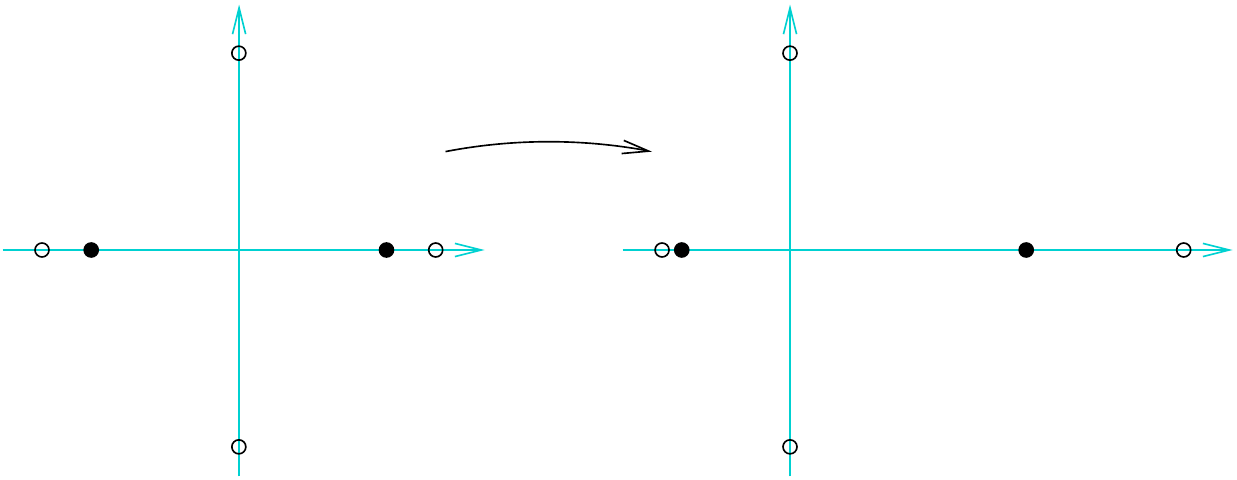}
\end{center}
\caption{Example \ref{exl:CounterManyMany2}: two subsets of the plane whose centroids can be fitted together in different ways.}
\end{figure}

\begin{exl}
\label{exl:CounterManyMany2}
Take the following subsets of the plane:
\[
K_1 = \left\{ \begin{pmatrix} -1\\ 0 \end{pmatrix}, \begin{pmatrix} 1\\ 0 \end{pmatrix}, \begin{pmatrix} 0\\ -1 \end{pmatrix}, \begin{pmatrix} 0\\ 1 \end{pmatrix} \right\}, \quad K_2 = \left\{ \begin{pmatrix} \frac{2}{\sqrt{7}}\\ 0 \end{pmatrix}, \begin{pmatrix} -\frac2{\sqrt{7}}\\ 0 \end{pmatrix} \right\}
\]
Again, both have centroid at the origin. Their images under the projective transformation $(x,y) \mapsto \left( \frac{x}{2-x}, \frac{y}{2-x}\right)$ both have centroids at $(\frac16, 0)$.
\end{exl}

\section{Point vs. a body}
\label{sec:PointBody}
\subsection{Cone sections and the Santal\'o point}
\label{sec:ConeSecSant}
A subset $C \subset \R^{d+1}$ is called a \emph{cone}, if $x \in C \Rightarrow \lambda x \in C\, \forall \lambda \ge 0$. A cone $C$ is \emph{pointed}, if $C \setminus \{0\}$ is contained in an open halfspace whose boundary hyperplane passes through the origin. A closed pointed cone possesses bounded sections by affine hyperplanes. We will consider only those sections that intersect each ray of the cone, and call them \emph{complete}. A pointed cone is the conical hull \eqref{eqn:ConeHull} of any of its complete sections.

For $K \subset \R^d$ put $C = \cone\{(1,x) \mid x \in K\}$.
Due to Propositions \ref{prp:A} and \ref{prp:B}, Theorem \ref{thm:OneVsBody} is equivalent to the following.

\setcounter{thmB}{1}
\begin{thmB}
\label{thm:OneVsBodyB}
Let $C \subset \R^{d+1}$ be a full-dimensional pointed cone with cusp-free affine sections. Then for every point $q \in \inn\conv C$ there exists a unique complete affine section of $C$ with centroid at~$q$.
\end{thmB}

\begin{exl}
For $d=1$ this means in particular that there is a unique chord of an angle $C$ that goes through a given point and has it as the midpoint. This is a popular elementary geometry problem. The endpoints of such a chord are found by intersecting $\partial C$ with its image under rotation by $180^\circ$ about the point. See Figure \ref{fig:PointInAngle}.

\begin{figure}[ht]
\begin{center}
\includegraphics{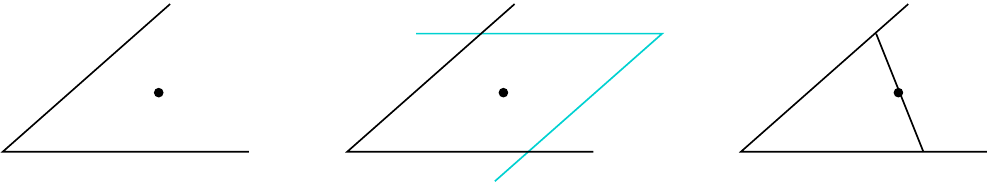}
\end{center}
\caption{Finding an angle chord bisected by a given point.}
\label{fig:PointInAngle}
\end{figure}
\end{exl}

Propositions \ref{prp:A} and \ref{prp:C} imply the following reformulation of Theorem~\ref{thm:OneVsBody} in the convex case.

\setcounter{thmC}{1}
\begin{thmC}
\label{thm:OneVsBodyC}
Let $L \subset \R^d$ be a convex $d$-dimensional body. Then there exists a unique point $y \in \inn L$ such that the polar dual of $L$ with respect to~$y$ has centroid at $y$.
\end{thmC}

The point $y$ in Theorem \ref{thm:OneVsBodyC} is called the Santal\'o point of $L$. Its existence and uniqueness was proved in \cite{San49}, see also \cite[pp. 165--166]{Leicht98}.

\subsection{Criticality of the volume}
Let $C \subset \R^{d+1}$ be a pointed full-dimensional cone. For every hyperplane $H$ such that $C \cap H$ is compact and intersects all rays of $C$, denote the bounded component of $C \setminus H$ by $C_H$.

\begin{prp}
\label{prp:Lever}
A hyperplane section $H \cap C$ of a cone $C$ has centroid at $q$ if and only if $H$ is a critical point of the function
\begin{equation}
\label{eqn:VolFunc}
H \mapsto \vol(C_H)
\end{equation}
on the set of all hyperplanes through $q$.
\end{prp}

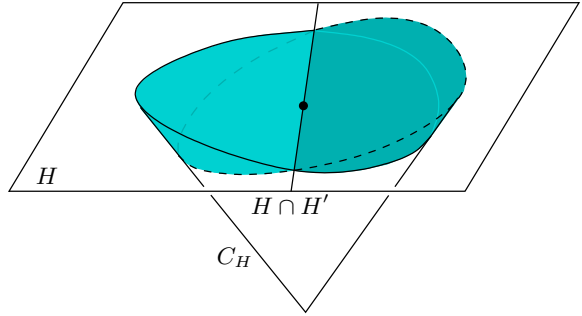
\begin{figure}[ht]
\begin{center}
\begin{picture}(0,0)%
\includegraphics{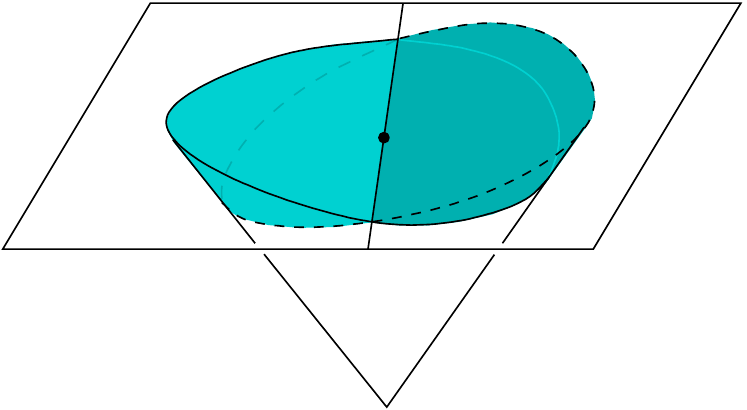}%
\end{picture}%
\setlength{\unitlength}{4144sp}%
\begingroup\makeatletter\ifx\SetFigFont\undefined%
\gdef\SetFigFont#1#2#3#4#5{%
  \reset@font\fontsize{#1}{#2pt}%
  \fontfamily{#3}\fontseries{#4}\fontshape{#5}%
  \selectfont}%
\fi\endgroup%
\begin{picture}(3399,1871)(-686,-1020)
\put(556,-722){\makebox(0,0)[lb]{\smash{{\SetFigFont{9}{10.8}{\rmdefault}{\mddefault}{\updefault}{\color[rgb]{0,0,0}$C_H$}%
}}}}
\put(757,-419){\makebox(0,0)[lb]{\smash{{\SetFigFont{9}{10.8}{\rmdefault}{\mddefault}{\updefault}{\color[rgb]{0,0,0}$H \cap H'$}%
}}}}
\put(-516,-243){\makebox(0,0)[lb]{\smash{{\SetFigFont{9}{10.8}{\rmdefault}{\mddefault}{\updefault}{\color[rgb]{0,0,0}$H$}%
}}}}
\end{picture}%
\end{center}
\caption{The volume difference between the differently shaded parts equals $\vol(C_H) - \vol(C_{H'})$. It vanishes in the first order if and only if $H \cap H'$ is an equilibrium axis for $C \cap H$.}
\label{fig:Lever}
\end{figure}

\begin{proof}
Consider two hyperplanes $H$ and $H'$ through $q$ close to each other. Then we have
\[
\vol(C_H) - \vol(C_{H'}) \approx \int\limits_{C \cap H} f(x)\, dx
\]
where $f \colon H \to \R$ is the linear function whose graph is $H'$. See Fig. \ref{fig:Lever}. Thus $H$ is critical for $\vol(C_H)$ if and only if all integrals over $C \cap H$ of linear functions vanishing at $q$ vanish.

On the other hand, vanishing of $\int_{C \cap H} f(x) \, dx$ for all linear functions $f$ with $f(q) = 0$ is equivalent to $q$ being the centroid of $C \cap H$. (Think of $f$ as the gravity torque with respect to the axis $\ker f$.)
\end{proof}


\begin{rem}
In 1931, Tricomi \cite{Tri31} and Guido Ascoli \cite{Asc31} showed that for every point inside a convex body there exists a hyperplane section that has this point as a centroid. Tricomi dealt only with dimension $3$ using the ``hairy ball theorem''. Ascoli used a variational approach based on Proposition \ref{prp:Lever}. They also characterized those non-convex bodies, for which the centroid of a section depends continuously on the hyperplane, making both approaches applicable. For more details see \cite[\S 2, Section 8]{BF87} that also deals with a beautiful related object, Dupin's ``floating body''.

Filliman \cite{Fil92} studied critical sections of polytopes and gave their characterization in the case of a simplex. 
\end{rem}

Our plan now is to show that for cones with cusp-free sections the function \eqref{eqn:VolFunc} is coercive, which implies the existence of a critical point, and then to prove some sort of convexity of \eqref{eqn:VolFunc} to ensure the uniqueness of the critical point.

\subsection{Logarithmic convexity of the volume and a proof of Theorem \ref{thm:OneVsBodyB}}
\label{sec:VolVar}
For a non-zero vector $y \in \R^{d+1}$ denote
\[
H_y \ce \{x \in \R^{d+1} \mid \langle x, y \rangle = 1\}
\]
The section $H_y \cap C$ is compact and intersects each ray of $C$ if and only if $y \in \inn C^\ast$, where
\[
C^\ast \ce \{y \in \R^{d+1} \mid \langle x, y \rangle \ge 0 \; \forall x \in C\}
\]
Note that if $C$ is $(d+1)$-dimensional, closed and pointed, then $C^\ast$ is also $(d+1)$-dimensional, closed and pointed.

For every $y \in \inn C^\ast$ denote
\[
C_y \ce \{x \in C \mid \langle x, y \rangle \le 1\}
\]
That is, $C_y$ is the bounded part of $C$ cut off by the hyperplane $H_y$.

%

Theorem \ref{thm:OneVsBodyB} is proved by using variational properties of the function
\begin{equation}
\label{eqn:FBody}
F \colon \inn C^\ast \to \R, \quad F(y) = \log \sqrt[d+1]{\vol(C_y)}
\end{equation}
The following arguments are a slight modification of \cite{Kloe13} and \cite{FK94}.

\begin{lem}
\label{lem:Momenta}
We have
\begin{gather*}
\int_C e^{-\langle x, y \rangle}\, dx = \frac{d!}{\|y\|} \vol_d(C \cap H_y) = (d+1)! \vol_{d+1}(C_y)\\
\int_C x e^{-\langle x, y \rangle}\, dx = \frac{(d+1)!}{\|y\|} \int_{C \cap H_y} x\, dx = (d+1)\cdot(d+1)! \vol_{d+1}(C_y) \bc(C\cap H_y)\\
\int_C \langle x, u \rangle^2 e^{-\langle x, y \rangle}\, dx = \frac{(d+2)!}{\|y\|} \int_{C\cap H_y} \langle x, u \rangle^2\, dx
\end{gather*}
\end{lem}
\begin{proof}
Represent $C$ as the union of parallel slices $C = \bigcup\limits_{t=0}^{+\infty} C\cap tH_y$. The distance between the hyperplanes $t_1 H_y$ and $t_2 H_y$ equals $\|y\|^{-1} |t_1 - t_2|$, therefore $\mu_{d+1} = \|y\|^{-1}\, dt\, \mu_d$, where $\mu_d$ is the $d$-dimensional Lebesgue measure on the hyperplanes orthogonal to $y$. Thus we have
\begin{multline*}
\int\limits_C e^{-\langle x, y\rangle}\, dx = \|y\|^{-1} \int\limits_0^{+\infty} \int\limits_{C \cap tH_y} e^{-\langle x, y \rangle}\, dx\, dt = \|y\|^{-1} \int\limits_0^{+\infty} \int\limits_{C \cap tH_y} e^{-t}\, dx\, dt\\
= \|y\|^{-1} \int\limits_0^{+\infty} \int\limits_{C \cap H_y} e^{-t} t^d\, dx\, dt = \|y\|^{-1} \vol_d(C \cap H_y) \int\limits_0^{+\infty} e^{-t} t^d\, dt\\
= d! \|y\|^{-1} \vol_d(C \cap H_y) = (d+1)! \vol_{d+1}(C_y)
\end{multline*}
because $C_y$ is a pyramid over $C \cap H_y$ with the altitude $\|y\|^{-1}$.

The second and the third integrals are computed similarly. Take into account that $\bc(C\cap H_y) = \frac{\int_{C\cap H_y} x\, dx}{\vol_d(C \cap H_y)}$.
\end{proof}

In particular, the function \eqref{eqn:FBody} equals
\[
F(y) = \frac{1}{d+1} \log \int_C e^{-\langle x, y \rangle}\, dx + \const
\]

\begin{lem}
\label{lem:GradF}
The gradient of $F$ is the centroid of the section:
\[
\grad F(y) = \gamma(C\cap H_y)
\]
\end{lem}
\begin{proof}
Let us compute the derivative of $F$ in the direction $u \in \R^{d+1}$:
\begin{multline*}
D_uF(y) = \frac{D_u (\int_C e^{-\langle x, y \rangle}\, dx)}{(d+1)\int_C e^{-\langle x, y \rangle}\, dx}\\
= \frac{\int_C \langle x, u \rangle e^{-\langle x, y \rangle}\, dx}{(d+1)\int_C e^{-\langle x, y \rangle}\, dx} = \left\langle u, \frac{\int_C x e^{-\langle x, y \rangle}\, dx}{(d+1)\int_C e^{-\langle x, y \rangle}\, dx} \right\rangle
\end{multline*}
Now the result follows from Lemma \ref{lem:Momenta}.
\end{proof}

\begin{lem}
\label{lem:InfinityBody}
If $K$ is cusp-free, then the value $F(y)$ tends to $+\infty$ as $y$ tends to a point in $\partial C^* \setminus \{0\}$.
\end{lem}
\begin{proof}
Let $y^0 \in \partial C^* \setminus \{0\}$.
Then there exists $x^0 \in C$ such that $\langle x^0, y^0 \rangle = 0$ and $x_0 \ne 0$. Clearly, $x^0 \in \partial C$. Thus, by assumption of Theorem \ref{thm:OneVsBody} there are vectors $x^1, \ldots, x^d \in \R^{d+1}$ such that their positive hull $\Delta = \{\sum_{i=0}^d \lambda_i x^i \mid \lambda_i \ge 0\}$ is contained in $C$. Then we have
\[
e^{(d+1)F(y)} \ge \int_\Delta e^{-\langle x, y \rangle}\, dx = \const \cdot \prod_{i=0}^d \int_0^\infty e^{-\lambda_i \langle x^i, y \rangle}\, d\lambda_i = \frac{\const}{\langle x^0, y \rangle \cdots \langle x^d, y \rangle}
\]
for some positive constant. Hence
\[
F(y) \ge -\frac1{d+1} {\sum_{i=0}^d \log \langle x^i, y \rangle} + \const
\]
As $y$ tends to $y^0$, the scalar product $\langle x^0, y \rangle$ tends to $0$ while other scalar products remain bounded below by a positive constant (some of them may also tend to $+\infty$). Hence $F(y) \to +\infty$.

%
\end{proof}

\begin{prp}
\label{prp:ConvexBody}
The function $F$ is strictly convex.
\end{prp}
\begin{proof}
We have
\[
D^2_{u,u}F(y) = \frac{\int\limits_{C\cap H_y} \langle x, u \rangle^2 e^{-\langle x, y \rangle}\, dx \cdot \int\limits_{C\cap H_y} e^{-\langle x, y \rangle}\, x - \left( \int\limits_{C\cap H_y} \langle x, u \rangle e^{-\langle x, y \rangle}\, dx \right)^2}{(d+1)\left( \int\limits_{C\cap H_y} e^{-\langle x, y \rangle}\, dx \right)^2}
\]
Using Lemma \ref{lem:Momenta}, we get
\begin{equation}
\label{eqn:SecDerBody}
D^2_{u,u}F(y) = (d+2) \frac{\int_{C\cap H_y} \langle x, u \rangle^2\, dx}{\vol_d(C \cap H_y)} - (d+1) \left( \frac{\int_{C\cap H_y} \langle x, u \rangle\, dx}{\vol_d(C\cap H_y)} \right)^2
\end{equation}
Due to the functional arithmetic-quadratic mean inequality
\[
\frac{\int_A f^2\, dx}{\vol(A)} \ge \left(\frac{\int_A f\, dx}{\vol(A)}\right)^2
\]
(which is the $L^2$ Cauchy-Schwarz inequality for functions $f$ and $1$) we have
\[
D^2_{u,u}F(y) \ge \frac{\int_{C \cap H_y} \langle x, u \rangle^2\, dx}{\vol_d(C \cap H_y)} > 0
\]
\end{proof}

\begin{proof}[Proof of Theorem \ref{thm:OneVsBodyB}]
The hyperplane $H_y$ passes through the point $q$ if and only if the hyperplane $H_q$ passes through $y$. The section $H_y \cap C$ is bounded if and only if $y \in \inn C^*$. Thus the hyperplane sections of $C$ coming into question are
\[
\{C_y \mid y \in \inn(C^* \cap H_q)\}
\]
Restrict the function $F$ defined in \eqref{eqn:FBody} to $C^*\cap H_q$. By Lemma \ref{lem:GradF} we have
\[
\grad F|_{C^* \cap H_q}(y) = \bc(C \cap H_y) - q
\]
(This is the projection of $\grad F$ to $H_q$; one may also evoke Lagrange multipliers.)
Thus $q$ is the centroid of $C\cap H_y$ if and only if $y$ is a critical point of $F|_{C^* \cap H_q}$.

Lemma \ref{lem:InfinityBody} implies that $F$ attains a minimum on $C^* \cap H_q$, which shows the existence part of Theorem \ref{thm:OneVsBodyB}. The uniqueness follows from Proposition \ref{prp:ConvexBody}, similarly to the proof of Theorem \ref{thm:OneVsManyA}.
\end{proof}

\begin{exl}
\label{exl:CounterPointBody}
Let $K = \{(x,y) \in \R^2 \mid -1 \le x \le 1, -(1-x)^3 \le y \le (1-x)^3\}$ and $q = (0,0)$. We claim that none of the maps
\[
\phi_{a,b} \colon (x,y) \mapsto \left( \frac{x}{1+ax+by}, \frac{y}{1+ax+by} \right)
\]
has the property $\bc(\phi_{a,b}(K)) = (0,0) = \phi_{a,b}(q)$. We have $\phi_{a,b} = \phi_{0,b} \circ \phi_{a,0}$. The set $\phi_{a,0}(K)$ is symmetric with respect to the $x$-axis, and it can be shown that for $b \ne 0$ the image under $\phi_{0,b}$ of an $x$-symmetric set has its centroid outside the $x$-axis. It follows that the only candidates for $\phi$ are the maps
\[
\phi_{a,0} \colon (x,y) \mapsto \left( \frac{x}{1+ax}, \frac{y}{1+ax} \right)
\]
For $\phi_{a,0}$ to be admissible, we have to assume $|a| > 1$.

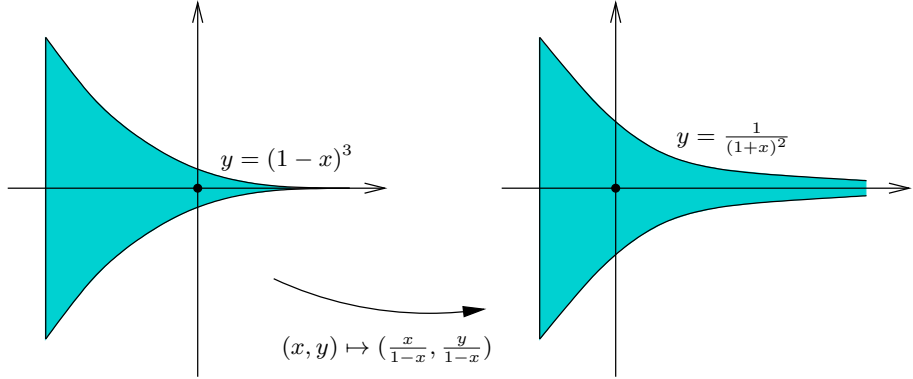
\begin{figure}[ht]
\begin{center}
\begin{picture}(0,0)%
\includegraphics{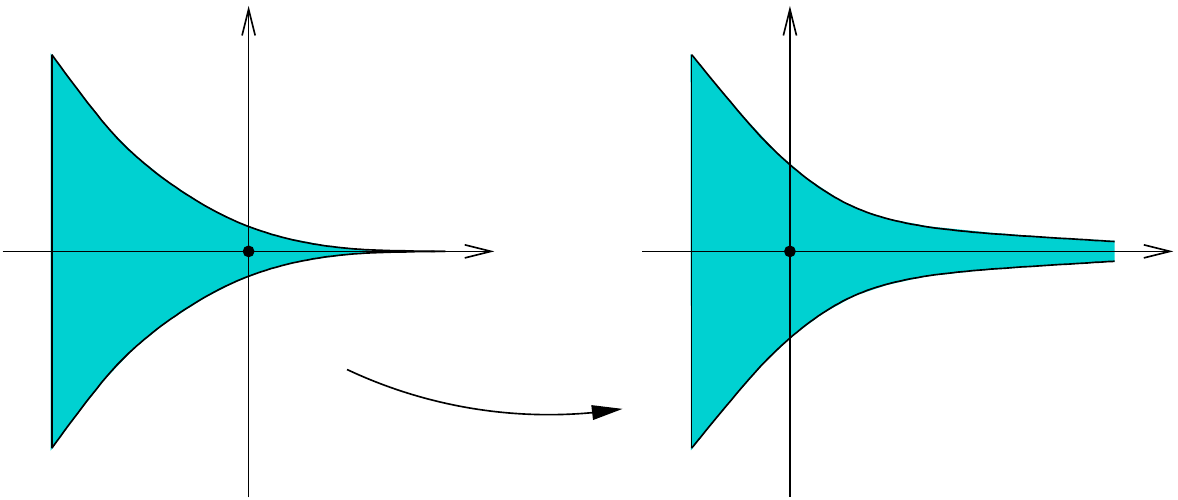}%
\end{picture}%
\setlength{\unitlength}{4144sp}%
\begingroup\makeatletter\ifx\SetFigFont\undefined%
\gdef\SetFigFont#1#2#3#4#5{%
  \reset@font\fontsize{#1}{#2pt}%
  \fontfamily{#3}\fontseries{#4}\fontshape{#5}%
  \selectfont}%
\fi\endgroup%
\begin{picture}(5379,2274)(-1136,-298)
\put(2836,1109){\makebox(0,0)[lb]{\smash{{\SetFigFont{9}{10.8}{\rmdefault}{\mddefault}{\updefault}{\color[rgb]{0,0,0}$y = \frac1{(1+x)^2}$}%
}}}}
\put(136,974){\makebox(0,0)[lb]{\smash{{\SetFigFont{9}{10.8}{\rmdefault}{\mddefault}{\updefault}{\color[rgb]{0,0,0}$y=(1-x)^3$}%
}}}}
\put(496,-151){\makebox(0,0)[lb]{\smash{{\SetFigFont{9}{10.8}{\rmdefault}{\mddefault}{\updefault}{\color[rgb]{0,0,0}$(x,y) \mapsto (\frac{x}{1-x}, \frac{y}{1-x})$}%
}}}}
\end{picture}%
\end{center}
\caption{The point that cannot become the centroid of a shaded figure under a projective transformation; even when the cusp is sent to infinity, the centroid lies to the left from the origin.}
\label{fig:CounterPointBody}
\end{figure}

A direct computation shows that the centroid of $\phi_{a,0}(K)$ always has a negative $x$-coordinate. In particular, in the limit case $a = -1$ we have
\[
\phi_{-1,0}(K) = \left\{(x,y) \left| x \ge -\frac12, - \frac1{(1+x)^2} \le y \le \frac1{(1+x)^2}\right.\right\}
\]
see Figure \ref{fig:CounterPointBody}, and
\[
\int_{-\frac12}^{+\infty} \frac{x}{(1+x)^2}\, dx = \log 2 - 2 < 0
\]
\end{exl}

\begin{rem}
Denote by $L = K^\circ$ the polar dual of $K$.
Propositions \ref{prp:C} and \ref{prp:B} imply
\[
\vol(L^\circ_y) = \vol(\phi_y(K)) = \|y\|^{-1} \vol(C \cap H_y) = (d+1) \vol_{d+1}(C_y)
\]
Therefore finding the minimum of $\vol_{d+1}(C_y)$ is equivalent to finding the minimum over all $y$ of the volume of the polar dual of $L$ with respect to $y \in \inn L$. This is the second characterization of the Santal\'o point of a convex body $L$, the first having been given in Theorem \ref{thm:OneVsBodyC}.

The maximum of the product $\vol(L) \vol(L^\circ)$ over all origin-symmetric convex bodies is achieved when $L$ is an ellipsoid. This is the
\emph{Blaschke-Santal\'o} inequality \cite{Bla17, San49, SR81}.
The minimum of $\vol(L) \vol(L^\circ)$ is not known, but is conjectured to be achieved when $L$ is a cube or cross-polytope or, more generally, Hanner polytopes (\emph{Mahler conjecture}).
\end{rem}

\section{Several points vs. a body}
\label{sec:ManyVsBody}
\subsection{Existence and non-uniqueness in the general case}
Here we prove Theorem \ref{thm:ManyVsBody}.

By combining the functionals \eqref{eqn:F1} and \eqref{eqn:FBody} we see that the classes of projective transformations satisfying \eqref{eqn:PhiBodyPoints} are in a $1$-to-$1$ correspondence with the critical points of the function
\begin{multline*}
F \colon \inn C^\ast \to \R, \quad F(y) = \frac1{d+1} \log\vol(C^-_y) + \frac1n \sum_{i=1}^n \log \langle p_i, y \rangle\\
= \frac1{d+1} \log\int_C e^{-\langle x, y \rangle}\, dx + \frac1n \sum_{i=1}^n \log \langle p_i, y \rangle + \const
\end{multline*}

If $p_i \in \inn\conv K$ for all $i$, then, for a cusp-free $K$, the integral tends to $+\infty$ as $y$ tends to $\partial C \setminus \{0\}$, while the sum remains bounded. This implies the existence if all $p_i$ lie in the interior of $\conv K$. If some of them lie on the boundary, then we need a more delicate argument.

\begin{lem}
\label{lem:InfinityBodySev}
If every support hyperplane of $K$ contains less than $\frac{n}{d+1}$ of the points $p_1, \ldots, p_n$, then the function $F$ tends to $+\infty$ as $y$ tends to a point in $\partial C^\ast \setminus\{0\}$.
\end{lem}
\begin{proof}
Let $y \to y_0 \in \partial C^\ast \setminus\{0\}$. As in the proof of Lemma \ref{lem:InfinityBody}, choose a point $x_0 \in C$ such that $\langle x_0, y_0 \rangle = 0$. Then we have
\[
\frac1{d+1} \log\vol(C^-_y) \ge - \frac1{d+1} \log \langle x_0, y \rangle + \const
\]
Now, we have $\langle p_i, y_0 \rangle \ge 0$. If for all $i$ this inequality is strict, then all $\log \langle p_i, y \rangle$ remain bounded as $y \to y_0$, so that $F(y) \to +\infty$.

Let $\langle p_1, y_0 \rangle = 0$ and $\langle p_i, y_0 \rangle > 0$ for $i \ne 1$. As $p_1 \in \conv(C)$, there exist $x_1, \ldots, x_k \in C$ such that
\[
p_1 = \sum_{i=1}^k \lambda_i x_i, \quad \lambda_i > 0
\]
This implies $\langle x_i, y_0 \rangle = 0$ for all $i$. It follows that
\[
\frac1{d+1} \log\int_C e^{-\langle x, y \rangle}\, dx \ge - \sum_{i=1}^k \log \langle x_i, y \rangle + \const
\]
On the other hand,
\[
\log \langle p_1, y \rangle = \log\left( \sum_{i=1}^n \lambda_i \langle x_i, y \rangle \right) \ge \log\lambda + \sum_{i=1}^n \frac{\lambda_i}{\lambda} \log \langle x_i, y \rangle
\]
where $\lambda = \sum_{i=1}^k \lambda_i$. Collecting all terms we get
\[
F(y) \ge  \left( - \frac1{d+1} + \sum_{i=1}^k \frac{\lambda_i}{n\lambda} \right)\log \langle x_i, y \rangle + \const
\]
Due to $\lambda_i \le \lambda$ and $n > d+1$, all coefficients before the logarithms are negative. Hence $F(y) \to +\infty$.

If the hyperplane $\langle x, y_0 \rangle = 0$ contains $m$ of the points $p_1, \ldots, p_n$, then $F(y)$ is bounded below by a sum of logarithms with coefficients $ - \frac1{m(d+1)} + \sum_i \frac{\lambda_i}{n\lambda}$, which are still negative provided that $m < \frac{n}{d+1}$.
\end{proof}

The restriction on the points lying on the boundary of the convex hull is necessary for the existence, as the following example shows.
\begin{exl}
\label{exl:d+1Points}
Let $K$ be the union of two $d$-simplices whose intersection is a $(d-1)$-face of both (a bipyramid), and let $p_i, i = 1, \ldots, d$ be the vertices of one of the simplices. Then the centroid of $\{p_1, \ldots, p_n\}$ coincides with the centroid of the corresponding simplex and therefore is different from the centroid of $K$. No projective transformation can help.

Alternatively, take $3$ points on one edge of the tetrahedron and $2$ points on the opposite edge. The centroid of the points lies on a plane parallel to both edges that divides the distance between them in proportion $2:3$. The centroid of the tetrahedron lies on a plane equidistant from both edges.
\end{exl}

\begin{figure}[ht]
\begin{center}
\includegraphics{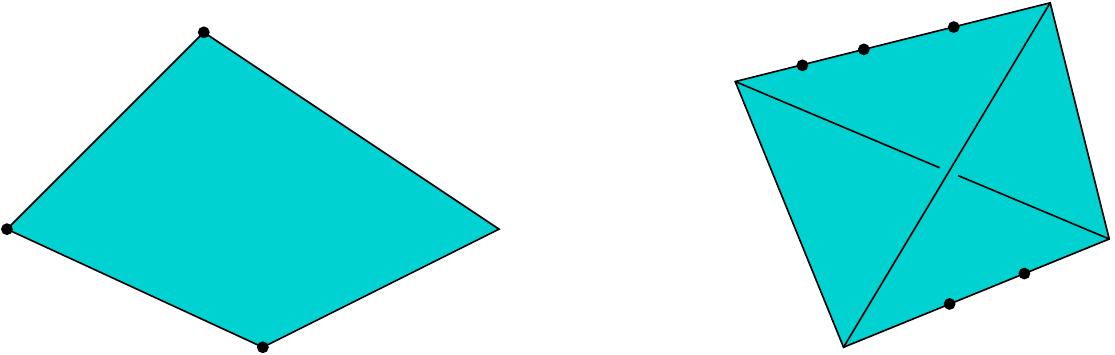}
\end{center}
\caption{Example \ref{exl:d+1Points}: no support hyperplane may contain too many $p_i$.}
\label{fig:CounterBodyMany1}
\end{figure}

In the following example the transformation $\phi$ is not unique.

\begin{exl}
\label{exl:CounterBodyMany}
Let $K \subset \R^2$ be the square with vertices $(\pm 1, 0)$, $(0, \pm 1)$, and
\[
p_1 = \left( -\frac{3}{\sqrt{13}}, 0 \right), \quad p_2 = \left( \frac{3}{\sqrt{13}}, 0 \right)
\]
Both $K$ and $\{p_1, p_2\}$ have centroid at the origin. The images of both sets under a projective non-affine transformation $(x,y) \mapsto \left( \frac{x}{3-x}, \frac{y}{3-x} \right)$ have centroids at $\left( \frac1{12}, 0 \right)$.

Compare this with Examples \ref{exl:CounterManyMany1} and \ref{exl:CounterManyMany2}. The coincidences are not accidental.
\end{exl}

\begin{figure}[ht]
\begin{center}
\includegraphics{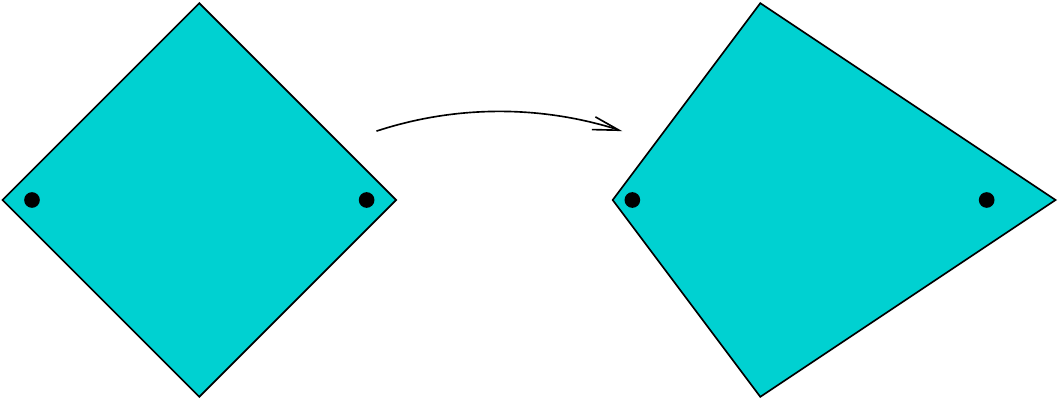}
\end{center}
\caption{A body and a pair of points whose centroids can be fitted in different ways.}
\label{fig:CounterBodyMany}
\end{figure}

\subsection{Centering of points inside a sphere}
Here we prove Theorem \ref{thm:ManyVsBall}. As usual, we form the difference of functions
\[
F \ce F_1 - F_2, \quad F_1(y) \ce \log \sqrt[d+1]{\vol(C_y)}, \quad F_2(y) \ce -\frac1n \sum_{i=1}^n \log \langle p_i, y \rangle
\]
where $C = \{y \in \R^{d+1} \mid y_0^2 \ge y_1^2 + \ldots + y_n^2,\, y_0 \ge 0\}$ is the cone over the unit ball. The domain of $F$ is the interior of the dual cone $C^*$ which coincides this time with $C$.

\begin{lem}
We have
\[
\vol(C_y) = \frac{\beta_n}{(d+1)\|y\|^{d+1}_{1,d}}
\]
where $\beta_n$ is the volume of a $d$-dimensional unit ball, and
\[
\|y\|_{1,d} \ce \sqrt{y_0^2 - y_1^2 - \ldots - y_n^2}
\]
is the Minkowski norm of $y$.
\end{lem}
\begin{proof}
As $\vol(C_y)$ is homogeneous of degree $-(d+1)$ with respect to $y$ (scaling $y$ by $\lambda$ results in scaling the truncated cone $C_y$ by $\lambda^{-1}$), it suffices to show that $\vol(C_y) = \frac{\beta_n}{d+1}$ if $\|y\|_{1,d} = 1$.

The equation $\langle x, y \rangle = 1$ of the hyperplane $H_y$ can be rewritten as $\langle x, \bar y \rangle_{1,d} = 1$ with $\bar y = (y_0, -y_1, \ldots, -y_n)$. It follows that for $\|y\|_{1,d} = 1$ the hyperplane $H_y$ is tangent to the upper half of the hyperboloid $\{\|y\|^2_{1,d} = 1\}$.
The group $\Ort(1,n)$ of linear transformations preserving the Minkowski scalar product acts transitively on the set of such hyperplanes, hence there is a transformation $f \in \Ort(1,n)$ that maps $H_y$ to the hyperplane $\{y_0 = 1\}$. Since $|\det f| = 1$, we have $\vol(C_y) = \vol(C \cap \{y_0 \le 1\})$. The latter is a cone of height $1$ over the unit ball and has volume $\frac{\vol_d(B^d)}{d+1}$.
\end{proof}
As a result we have
\begin{equation}
\label{eqn:FBall}
F(y) = -\log \|y\|_{1,d} + \frac1n \sum_{i=1}^n \log \langle p_i, y \rangle + \const
\end{equation}

By results of Section \ref{sec:HomogPoints} and Lemma \ref{lem:GradF},
\[
\grad F(y) = \bc(C_y) - \bc(H_y \cap R_1, \ldots, H_y \cap R_n)
\]
where $R_i$ is the ray generated by $p_i$. Therefore we have to show that the function $F$ has a unique, up to scaling, critical point.
Note also that $F(\lambda y) = F(y)$, so that it suffices to consider the restriction of $F$ to any subset that is represented in all equivalence classes $y \sim \lambda y$. Two convenient choices are $\{y \mid y_0 = 1\}$ and $\{y \mid \|y\|_{1,d} = 1\}$.

\begin{lem}
\label{lem:CoercBall}
The function $F$ tends to $+\infty$ as $y$ tends to a point in $\partial C \setminus\{0\}$.
\end{lem}
\begin{proof}
Let $y \to z$ with $\|z\|_{1,d} = 0$. Then $-\log \|y\|_{1,d} \to +\infty$. If $\langle p_i, z \rangle \ne 0$ for all $i$, then the other summands in \eqref{eqn:FBall} remain bounded, and the sum tends to $+\infty$.

If there is an $i$ such that $\langle p_i, z \rangle_{1,d} = 0$, then $p_i=z$, so that only the $i$-th summand under the sum sign in \eqref{eqn:FBall} tends to $-\infty$. We then have
\[
F(y) = - \frac12 \log t + \frac1n \log t + O(1) \to +\infty
\]
where $t = \langle y, z \rangle_{1,d}$.
\end{proof}

This already implies the existence of a critical point of $F$. For uniqueness we would like to use convexity, but the following example shows that $F$ is not always convex.

\begin{exl}
For $d \ge 3$ put $p_i = e_0 + ae_1 + b_ie_2$ and consider the restriction of $F(y)$ to the line $y = e_0 + te_1$. There we have
\[
F(y) = -\log \sqrt{1-t^2} + \log(1+at) = \frac12 \log\frac{(1+at)^2}{1-t^2}
\]
For $a = 0.95$ this function is not convex.
\end{exl}

The following trick helps.

\begin{lem}
\label{lem:GeodConv}
The function $F$ is geodesically strictly convex with respect to the hyperbolic metric on $\inn C/\{x \sim \lambda x\}$.
\end{lem}
\begin{proof}
When restricted to $\{\|y\|_{1,d} = 1\}$, the function $F$ has the form
\begin{equation}
\label{eqn:FHyp}
F(y) = \sum_{i=1}^n \log\langle p_i, y \rangle = \sum_{i=1}^n \log \langle \bar{p_i}, y \rangle_{1,d}
\end{equation}
Every geodesic is represented by a hyperplane section of the hyperboloid $\{\|y\|^2_{1,d} = 1\}$, and has a unit speed parametrization of the form
\[
y(t) = q \cosh t + r \sinh t
\]
where $\|q\|^2_{1,d} = 1$, $\|r\|^2_{1,d} = -1$, $\langle q, r \rangle_{1,d} = 0$. 
Let us study the restrictions of the $i$-th summand in \eqref{eqn:FHyp} to geodesics.

If $\|p_i\|^2_{1,d} > 0$, then on any geodesic it is possible to choose $q$ and $r$ so that $\langle \bar{p_i}, r \rangle_{1,d} = 0$. We get
\begin{equation}
\label{eqn:LogCosh}
\log \langle \bar{p_i}, y \rangle_{1,d} = \log \cosh t + \const
\end{equation}
which is strictly convex.

If $\|p_i\|^2_{1,d} = 0$ and the geodesic doesn't have $\bar{p_i}$ as a limit point, then one can do the same.

If $\|p_i\|^2_{1,d} = 0$ and the geodesic has $\bar{p_i}$ as a limit point, then for any parametrization we have $\langle \bar{p_i}, q \rangle_{1,d} = -\langle \bar{p_i}, r \rangle_{1,d}$, so that
\begin{equation}
\label{eqn:Horo}
\log \langle \bar{p_i}, y \rangle_{1,d} = \log(\cosh t + \sinh t) + \const = t + \const
\end{equation}
Thus along such a geodesic the function is linear.

The only possibility for the sum \eqref{eqn:FHyp} to be linear along a geodesic (and thus non-strictly convex) is that all points $\{p_i\}$ have $\|p_i\|^2_{1,d} = 0$ and lie on that geodesic. This is only possible for $n \le 2$.
\end{proof}

\begin{proof}[Proof of Theorem \ref{thm:ManyVsBall}]
By Lemma \ref{lem:CoercBall}, the function $F$ has a critical point inside the cone $C$. By Lemma \ref{lem:GeodConv}, this critical point is unique up to scaling, because otherwise the restriction of $F$ to a geodesic would have two different critical points, which contradicts the strict geodesic convexity of~$F$.
\end{proof}

\begin{rem}
\label{rem:Moebius}
This argument generalizes that of Springborn \cite{Spr05}, who considers only points on the sphere. In this case, the function $F$ is the sum of hyperbolic distances to horospheres centered at the given points. For a point $p$ inside the ball, the term $\log \langle p, y \rangle$ equals $\log \cosh \dist(p', y)$, where $\dist$ is the hyperbolic distance, and $p' = \bar{p}/\|p\|_{1,d}$.

In the case when all $p_i$ lie on the sphere, the critical point of the function $F$ is the so called conformal barycenter of $\{p_i\}$. In \cite{DE86}, the conformal barycenter was defined for non-atomic measures on the sphere, and the construction for discrete measures was indicated.

The ``centroid'' of points in the hyperbolic space can be defined in different ways. One of the possibilities is to take the affine centroid of the points on the hyperboloid and centrally project back; this point minimizes $\sum_i \cosh \dist(x, p_i)$. Another possibility is to minimize $\sum_i \dist^2(x, p_i)$ as in the general definition of the Riemannian center of mass \cite{GK73}.
\end{rem}

%

\section{Fitting centroids of two bodies}
\label{sec:TwoBodies}
\subsection{Existence and non-uniqueness}
\label{sec:ExistTwoBodies}
We approach Theorem \ref{thm:BodyVsBody} in Reformulation \ref{ref:B}: take cones $C_1$ and $C_2$ over $K_1$ and $K_2$, respectively, and see under what conditions there is an affine hyperplane $H$ such that $C_1 \cap H$ and $C_2 \cap H$ have a common centroid.

Following Section \ref{sec:VolVar}, introduce the functions
\[
F_i \colon \inn C^*_i \to \R, \quad F_i(y) \ce \log \sqrt[d+1]{\vol(C_{i,y})}, \quad i = 1, 2
\]
where $C_{i,y} = \{x \in C_i \mid \langle x, y \le 1\}$ is the cone $C_i$ truncated by the hyperplane~$H_y$. Their difference
\begin{equation}
\label{eqn:FTwoBodies}
F \colon \inn C^*_1 \to \R, \quad F \ce F_1 - F_2
\end{equation}
has, according to Lemma \ref{lem:GradF}, the gradient
\[
\grad F(y) = \bc(C_1 \cap H_y) - \bc(C_2 \cap H_y)
\]
Thus the following lemma holds.

\begin{lem}
Projective transformations that fit the centroids of $K_1$ and $K_2$ correspond, modulo post-composition with affine transformations, to critical points of the function $F$ from \eqref{eqn:FTwoBodies}.
\end{lem}

The existence of a critical point follows by the usual argument.

\begin{proof}[Existence part of Theorem \ref{thm:BodyVsBody}]
By Lemma \ref{lem:InfinityBody}, the function $F_1(y)$ tends to $+\infty$ as $y$ tends to $\partial C^*_1 \setminus \{0\}$. The function $F_2$ is continuous on $\inn C^*_2 \supset C^*_1 \setminus\{0\}$, and therefore bounded on $\inn C^*_1$. Thus $F(y) \to +\infty$ as $y \to \partial C^*_1 \setminus \{0\}$.
\end{proof}

As next we give an example where the projective transformation fitting the centroids is not unique.
\begin{exl}
\label{exl:CounterBodyBody}
Take a unit disk and the following rectangle inside it:
\[
K_1 = \{(x,y) \in \R^2 \mid x^2 + y^2 \le 1\}, \quad K_2 = \{(x,y) \in \R^2 \mid |x| \le a, |y| \le \frac{2}{\sqrt{5}}\}
\]
for some $a < \frac1{\sqrt{5}}$. Both $K_1$ and $K_2$ have centroid at the origin. On the other hand, the projective transformation
\[
(x,y) \mapsto \left( \frac{\sqrt{2} x}{y + \sqrt{3}}, \frac{\sqrt{3} y + 1}{y + \sqrt{3}} \right)
\]
maps the disk $K_1$ to itself, and the rectangle $K_2$ to a trapezoid that, as a tedious computation shows, also has centroid at the origin.
See Fig.~\ref{fig:CounterTwoBodies}.

An example of this sort is possible whenever the rectangle has a side which is longer than $\sqrt{3}$.


\begin{figure}[ht]
\begin{center}
\includegraphics{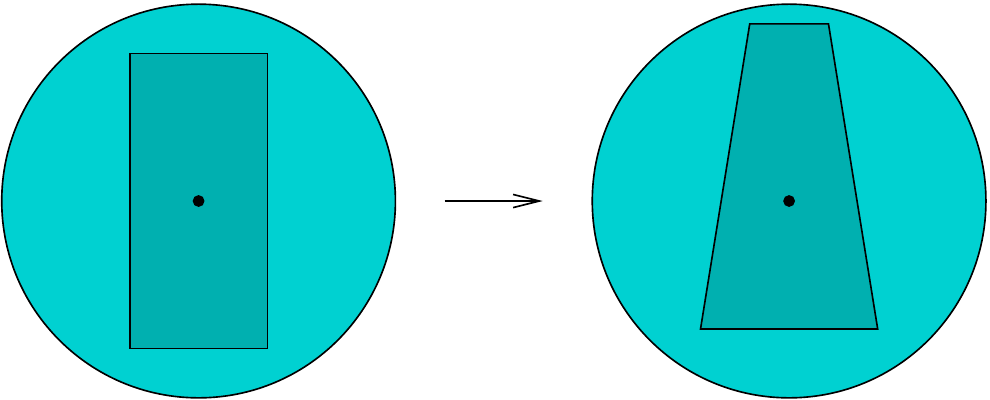}
\end{center}
\caption{Two bodies whose centroids can be fitted in different ways.}
\label{fig:CounterTwoBodies}
\end{figure}
%
\end{exl}

\subsection{Uniqueness for $K_2$ deep inside $K_1$}
We will consider the restriction of $F$ to the section of $\inn C_1^*$ by the hyperplane $y_0 = 1$. Since $F(\lambda y) = F(y)$, every critical point of this restriction is a critical point of $F$. We are not able to prove that $F|_{y_0=1}$ is convex under assumption \eqref{eqn:CR}, but we can prove that it is strictly convex at every critical point. Since the indices of critical points of a function defined on a ball and tending to $+\infty$ near boundary sum up to $0$, this implies that the critical point is unique.

\begin{lem}
\label{lem:D2F}
Let $y \in \R^{d+1}$ be such that $y_0 = 1$. Then for every vector $u \in \R^{d+1}$ with $u_0=0$ we have
\begin{multline*}
D^2_{u,u}F(y) = (d+2) \left( \frac{I_u(\phi_y(K_1))}{\vol(\phi_y(K_1))} - \frac{I_u(\phi_y(K_2))}{\vol(\phi_y(K_2))} \right)\\
+ \langle \bc(\phi_y(K_1)), u \rangle^2 - \langle \bc(\phi_y(K_2)), u \rangle^2
\end{multline*}
where $I_u(K) \ce \int_K \langle x-\bc(K), u \rangle^2\, dx$ is the moment of inertia of $K$ with respect to the vector $u$.
\end{lem}
\begin{proof}
From \eqref{eqn:SecDerBody} we have
\[
D^2_{u,u}F_i(y) = (d+2) \frac{\int_{C_i\cap H_y} \langle x, u \rangle^2\, dx}{\vol_d(C_i \cap H_y)} - (d+1) \langle \bc(C_i \cap H_y), u \rangle^2
\]
By Proposition \ref{prp:B}, $\phi_y(K_i)$ is the image of $C_i \cap H_y$ under parallel projection along $e_0$. Therefore
\begin{gather*}
\int_{C_i \cap H_y} \langle x, u \rangle^2\, dx = \|y\| \int_{\phi_y(K_i)} \langle x, u \rangle^2\, dx \\
\vol_d(C_i \cap H_y) = \|y\| \vol_d(\phi_y(K_i))
\end{gather*}
Also, $\bc(\phi_y(K_i))$ differs from $\bc(C_i \cap H_y)$ by a multiple of $e_0$. It follows that for $u_0=0$ we can replace $C_i \cap H_y$ by $\phi_y(K_i)$ in the formula for $D^2F_i$.

Further, for any $K \subset \R^d$ and $\bc = \bc(K)$ we have
\begin{multline*}
\int_K \langle x - \bc, u \rangle^2\, dx = \int_K \langle x, u \rangle^2\, dx - 2 \langle \bc, u \rangle \int_K \langle x, u \rangle\, dx + \int_K \langle \bc, u \rangle^2\, dx\\
= \int_K \langle x, u \rangle^2\, dx - \vol(K) \langle \bc, u \rangle^2
\end{multline*}
By substituting this into the last equation we obtain
\[
D^2_{u,u}F_i(y) = (d+2) \frac{I_u(\phi_y(K_i))}{\vol_d(\phi_y(K_i))} + \langle \bc(\phi_y(K_i)), u \rangle^2
\]
and the lemma follows.
\end{proof}

\begin{lem}
\label{lem:MomWidth}
Let $K \subset \R^d$ be a convex body, and $u \in \R^d$ be a non-zero vector. Then we have
\[
\frac{W_u(K)^2}{2(d+1)(d+2)} \le \frac{I_u(K)}{\vol(K)} \le \frac{W_u(K)^2}{12}
\]
where $W_u(K) = \frac{1}{\|u\|}(\max_{x\in K} \langle x, u \rangle - \min_{x\in K} \langle x, u \rangle)$ is the width of $K$ in the direction of $u$. The lower bound is achieved for a symmetric double cone over any $(d-1)$-dimensional body, the upper bound is achieved for the cylinder over any $(d-1)$-dimensional body.
\end{lem}
\begin{proof}
If each section of $K$ orthogonal to $u$ is replaced by a $(d-1)$-ball of the same radius, then the body remains convex and preserves its volume and moment of inertia in the direction $u$. Thus, without loss of generality, $K$ is a ``rotation body'' with axis $u$.

We will use the fact that moving mass away from the centroid increases the moment of inertia, and moving towards decreases the moment.

By the above principle, the Steiner symmetrization with respect to $u^\perp$ preserves the volume but decreases the moment. The resulting body is symmetric with respect to a hyperplane orthogonal to $u$, and it is possible to move more mass towards the centroid by replacing each of the symmetric halves by a cone of the same volume with the base on the hyperplane of symmetry. The two steps are illustrated on Fig. \ref{fig:MomentCone}. The convex profiles stand for the radii of the sections orthogonal to $u$; equally colored regions correspond to sets of equal $d$-volume.

\begin{figure}[ht]
\begin{center}
\includegraphics{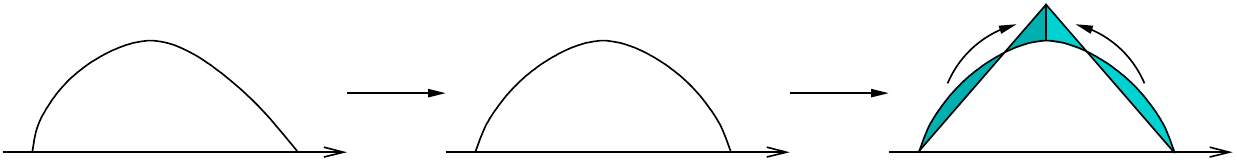}
\end{center}
\caption{Minimizing the moment of inertia for fixed width and volume.}
\label{fig:MomentCone}
\end{figure}

For a double cone of width $2a$ we have
\[
\frac{I_u(K)}{\vol(K)} = \frac{\int_0^a (a-t)^2t^{d-1}\, dt}{\int_0^a t^{d-1}\, dt} = \frac{2a^2}{(d+1)(d+2)} = \frac{W_u(K)^2}{2(d+1)(d+2)}
\]
which yields the lower bound in the theorem.

In order to prove the upper bound, first replace $K$ with a truncated cone $K_1$ whose parts on either sides from the hyperplane through the centroid of $K$ have the same volumes as the corresponding parts of $K$. One can go from $K$ to $K_1$ by moving mass away from the centroid, see Figure \ref{fig:MomentCyl}, therefore $K_1$ has a bigger moment of inertia.

\begin{figure}[ht]
\begin{center}
\includegraphics[width=.95\textwidth]{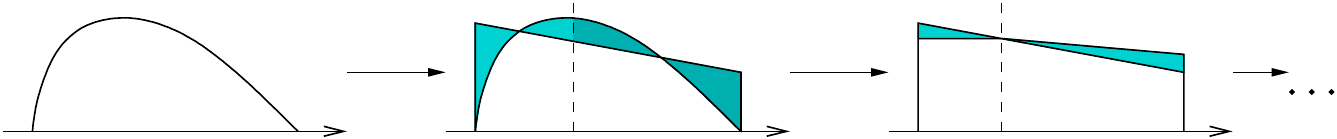}
\end{center}
\caption{Maximizing the moment of inertia for fixed width and volume.}
\label{fig:MomentCyl}
\end{figure}

It turned out unexpectedly hard to prove directly that the cylinder maximizes the moment of inertia among all truncated cones of a fixed volume, therefore we will continue to move mass. We replace $K_1$ by a union $K_2$ of a cylynder and a truncated cone as shown on Figure \ref{fig:MomentCyl}. A direct computation shows that the requirement $\vol(K_2 \setminus K_1) = \vol(K_1 \setminus K_2)$ leads to a convex $K_2$ (the radius of the cone decreasing as on Figure \ref{fig:MomentCyl}). Also, the section of $K_2 \setminus K_1$ by a hyperplane orthogonal to $u$ has a smaller volume as the section of $K_1 \setminus K_2$ at the same distance from the centroid of $K_1$. This allows to map $K_1 \setminus K_2$ to $K_2 \setminus K_1$ so that the mass is moved away from the centroid of $K_1$. Thus $I_u(K_2) \ge I_u(K_1)$.

The body $K_2$ can be replaced by a truncated cone $K_3$ with a bigger moment, as it was done at the first step. By iterating the procedure, we obtain a sequence of bodies converging to a cylinder. This implies that the cylinder maximizes the moment of inertia for given volume and width.

The ratio for the cylinder of width $2a$ equals
\[
\frac{I_u(K)}{\vol(K)} = \frac{\int_0^a t^2\, dt}{a} = \frac{a^2}{3} = \frac{W_u(K)^2}{12}
\]
\end{proof}

\begin{lem}
\label{lem:CRWidth}
Let $a_1 < a_2 < b_2 < b_1$. Then for every $\kappa \in (0,1)$ we have
\[
\crra(b_2, a_2; a_1, b_1) < \frac{(1+\kappa)^2}{(1-\kappa)^2} \Rightarrow \frac{b_2 - a_2}{b_1 - a_1} < \kappa
\]
\end{lem}
\begin{proof}
For a fixed cross-ratio, the maximum of $b_2 - a_2$ is achieved when the segments are concentric, that is $x=-y$. We have
\[
\crra(y,-y;-x,x) = \frac{y+x}{-y+x} : \frac{y-x}{-y-x} = \frac{(x+y)^2}{(x-y)^2} = \frac{(1+\frac{y}{x})^2}{(1-\frac{y}{x})^2}
\]
and the lemma follows.
\end{proof}

\begin{lem}
\label{lem:Morse}
Let $D \subset \R^d$ be homeomorphic to the ball $B^d$, and $F \colon \inn D \to \R$ be a smooth function on the interior of $D$. Assume that $F(y) \to +\infty$ as $y \to \partial D$ and that $D^2F(y)$ is positive definite at every critical point. Then the critical point of $F$ is unique.
\end{lem}
\begin{proof}
The function $F$ attains its minimum in $D$, therefore it has at least one critical point.
Due to $D^2F(y) > 0$ at all critical points, critical points are isolated. Due to $F(y) \to +\infty$ as $y \to \partial D$, the critical values form a discrete subset of $\R$, and in particular can be ordered:
\[
F(y^1) \le F(y^2) \le \cdots \le F(y^k) \le \cdots
\]
Let $a \in (F(y^k), F(y^{k+1}))$. By the Morse theory, the set $F^{-1}(-\infty, a)$ is homeomorphic to the union of $k$ open disks, with $y_1, \ldots, y_k$ lying in different disks.

On the other hand, if we choose a path $\alpha \colon [0,1] \to D$ joining $y_1$ and $y_2$ and take
\[
a > \max_{y \in \alpha[0,1]} F(y)
\]
then $y_1$ and $y_2$ lie in the same component of $F^{-1}(-\infty, a)$. This contradiction shows that the critical point is unique.
\end{proof}

\begin{proof}[Proof of uniqueness in Theorem \ref{thm:BodyVsBody}]
We are considering the restriction of $F$ to $\inn(C_1^* \cap \{y_0 = 1\}) = \inn K_1^\circ$.
From Section \ref{sec:ExistTwoBodies} we know that $F(y) \to +\infty$ as $y \to \partial K_1^\circ$.
Let us show that under assumption \eqref{eqn:CR} the quadratic form $D^2F$ is positive definite at all critical points.
Due to Lemma \ref{lem:D2F}, at a critical point we have
\[
D^2_{u,u}F(y) = (d+2) \left( \frac{I_u(\phi_y(K_1))}{\vol(\phi_y(K_1))} - \frac{I_u(\phi_y(K_2))}{\vol(\phi_y(K_2))} \right)
\]
Consider the orthogonal to~$u$ support hyperplanes $a_1, b_1$ and $a_2, b_2$ of $\phi_y(K_1)$ and $\phi_y(K_2)$. They are images under $\phi_y$ of support hyperplanes of $K_1$ and $K_2$ that are either parallel or share a $(d-2)$-dimensional affine subspace. Since the cross-ratio is projectively invariant, \eqref{eqn:CR} holds for $a_1, b_1$ and $a_2, b_2$.
By Lemma \ref{lem:CRWidth} we have
\[
\frac{W_u(\phi_y(K_2))}{W_u(\phi_y(K_1))} < \kappa_d = \sqrt{\frac{6}{(d+1)(d+2)}}
\]
which, by Lemma \ref{lem:MomWidth}, implies
\[
\frac{I_u(\phi_y(K_2))}{\vol(\phi_y(K_2))} < \frac{I_u(\phi_y(K_1))}{\vol(\phi_y(K_1))}
\]
Thus, at every critical point $D^2_{u,u}F(y) > 0$ for all $u \ne 0$. By Lemma \ref{lem:Morse}, this implies that the critical point is unique.
\end{proof}

Let us show that the bound \eqref{eqn:CR} is sharp. Take as $K_1$ a double cone of height $1$ over a $(d-1)$-dimensional subset of $\R^{d-1}$ with centroid at the origin, and as $K_2$ a cylinder of height $> \kappa_d$ over a similar, but smaller, set. For example, $K_1$ may be the standard cross-polytope, and $K_2$ a rectangular parallelepiped. Then the centroids of $K_1$ and $K_2$ coincide, so that $e_0$ is a critical point of the function $F$. The quadratic form $D^2F(e_0)$ takes a negative value in the direction of the axis of $K_1$ and $K_2$. Therefore $e_0$ is not the minimum point of $F$. Thus a minimum point provides a non-affine projective transformation that fits the centroids of $K_1$ and $K_2$.

\subsection{If one of the bodies is a ball}

\begin{lem}
\label{lem:MomBall}
For a $d$-dimensional ball $B^d_r$ of radius $r$ we have
\[
\frac{I_u(B^d_r)}{\vol(B^d_r)} = \frac{r^2}{d+2}
\]
\end{lem}
\begin{proof}
The moment $I_u(B^d_r)$ doesn't depend on $u$. The sum of the moments in $d$ pairwise orthogonal directions equals the polar moment $\int_{B^d_r} \|x\|^2\, dx$. Thus we have
\[
I_u(B^d_r) = \frac1{d} \int_0^r \int_{S^{d-1}_t} t^2\, dx\, dt = \frac{\omega_{d-1}}{d} \int_0^r t^{d+1}\, dt = \frac{\omega_{d-1}r^{d+2}}{d(d+2)} 
\]
where $S^{d-1}_t$ is the $(d-1)$-dimensional sphere of radius $t$, and $\omega_{d-1}$ is the volume of the unit $(d-1)$-sphere. On the other hand, $\vol(B_r) = \frac1{d} \omega_{d-1} r^d$, which leads to the formula of the lemma.
\end{proof}

\begin{proof}[Proof of Theorem \ref{thm:BodyVsBall}]
Similar to the proof of Theorem \ref{thm:BodyVsBody}, it suffices to show that
\[
\frac{I_u(\phi_y(B^d))}{\vol(\phi_y(B^d))} > \frac{I_u(\phi_y(K))}{\vol(\phi_y(K))}
\]
for all $y$.

The image of a ball under an admissible projective transformation is an ellipsoid. It is easily seen that the normalized moment of inertia $I_u/\vol$ of an ellipsoid equals to the normalized moment of a ball with diameter equal to the width of the ellipsoid in direction $u$:
\[
\frac{I_u(\phi_y(B^d))}{\vol(\phi_y(B^d))} = \frac{W_u(\phi_y(B^d))^2}{4(d+2)}
\]
Because of Lemmas \ref{lem:MomWidth} and \ref{lem:MomBall} we have
\begin{multline*}
\frac{I_u(\phi_y(K))}{\vol(\phi_y(K))} < \frac{W_u(\phi_y(B^d))^2}{4(d+2)} \Leftarrow \frac{W_u(\phi_y(K))^2}{12} < \frac{W_u(\phi_y(B^d))^2}{4(d+2)}\\
\Leftarrow \frac{W_u(\phi_y(K))}{W_u(\phi_y(B^d))} < \sqrt{\frac{3}{d+2}}
\end{multline*}
To ensure the latter inequality for all $u$, it suffices to require
\[
\crra(b_2, a_2; a_1, b_1) < \frac{(1 + \sqrt{\frac{3}{d+2}})^2}{(1 - \sqrt{\frac{3}{d+2}})^2}
\]
for all quadruples of parallel tangent hyperplanes to $\phi_y(B^d)$ and $\phi_y(K)$. This, in turn, is implied by the same inequality for concurrent tangent hyperplanes to $B^d$ and $K$.
\end{proof}

\begin{proof}[Proof of Theorem \ref{thm:BallVsBody}]
Without loss of generality, assume $0 \in \inn K$ (this may be achieved by a projective transformation that fixes $B$, and the Theorem is of projective nature).
By Section \ref{sec:CompPolar}, $K_y^\circ = \phi_y(K^\circ) + y$, so that
\[
\bc(K^\circ_y) = \bc(B^\circ_y) \Leftrightarrow \bc(\phi_y(K^\circ)) = \bc(\phi_y(B))
\]
Use the same method as in the proofs of Theorems \ref{thm:BodyVsBody} and \ref{thm:BodyVsBall}. The assumption of the theorem implies
\[
\width(K^\circ) < \log \frac{1+\sqrt{\frac{2}{d+1}}}{1-\sqrt{\frac{2}{d+1}}}
\]
which implies
\[
\frac{W_u(\phi_y(K^\circ))}{W_u(\phi_y(B))} < \sqrt{\frac{2}{d+1}}
\]
for all $y$, and in particular for critical points of the function $F$. Due to Lemmas \ref{lem:MomWidth} and \ref{lem:MomBall}, this implies that at the critical points $F$ is strictly convex. Thus by Lemma \ref{lem:Morse} the critical point is unique.
\end{proof}

\section{Open questions}
\label{sec:FutRes}
\subsection{Other dimensions}
We restricted our attention to the cases $\dim K_i \in \{0, d\}$ because it implies the affine covariance of the centroid \eqref{eqn:AffCovar}, which makes it possible to formulate the uniqueness problem (projective transformations modulo affine ones). For $\dim K = k \notin \{0, d\}$ the centroid is affinely covariant under some additional restrictions, for example if $K$ is centrally symmetric or if the affine span of $K$ has dimension $k$.

\begin{prb}
What is the most general class of $k$-dimensional subsets of $\R^d$ with affinely covariant centroids?
\end{prb}

Note that the affine covariance of the centroid of $K$ doesn't yet make the uniqueness question well-posed: in principle, one needs the affine covariance for all projective images of $K$.

\begin{prb}
Study the existence and uniqueness questions when $\dim K_i \notin \{0, d\}$ for at least one $i$. Do the solutions correspond to the critical points of some function $F$?
\end{prb}

\subsection{Point vs a body}
By Theorem \ref{thm:OneVsBody}, every point inside a cusp-free set becomes the centroid after some projective transformation.
Example \ref{exl:CounterPointBody} gives a point inside a set lying close to a sharp cusp so that no projective transformation can make it the centroid.

\begin{prb}
Weaken the cusp-freeness condition so that to preserve the existence of a projective transformation for any point in the interior. Is the following condition necessary and sufficient: the image of $K$ under any projective transformation that sends some support hyperplane of $K$ to infinity has infinite volume?
\end{prb}

\begin{prb}
If $K$ has sharp cusps, describe the set of all points that can become centroids. Is it convex? Is it related to the Dupin's floating body?
\end{prb}

We conjecture the following solution of the latter problem: for every support hyperplane $\ell$ of $K$ take a projective transformation $\phi_\ell$ that sends $\ell$ to infinity. If $\phi_\ell(K)$ has a finite volume, then cut off by a hyperplane parallel to $\ell$ half of the volume of $\phi_\ell(K)$ (remove the infinite part). The complement in $K$ to the preimages of all parts removed in this way is the set of points that can become the centroid.

\subsection{Several points vs a body}
In Examples \ref{exl:CounterBodyMany} and \ref{exl:d+1Points} we saw that the projective transformation fitting the centroids of a finite set and of a ``bigger'' $d$-dimensional set does not always exist, and if, then may be not unique. By contrast, if the $d$-dimensional set is a ball, then we have unconditional existence and uniqueness.

\begin{prb}
Does a projective transformation always exist, if through every point on the boundary of $\conv K$ goes exactly one support hyperplane?
\end{prb}

\begin{prb}
Is the projective transformation unique, if $K$ is ``round enough'' in some sense?
\end{prb}

\subsection{If no body is contained in the other}
Our method to prove the existence was to show that the function $F$ tends to $+\infty$ near the boundary of its domain. This was ensured by the assumption $K_2 \subset \inn\conv K_1$. It is possible to prove the existence of a critical point of $F$ under less restrictive assumptions, for example, if $d=2$ and the gradient of $F$ ``turns'' as we go along the boundary of domain of $F$. This means that a projective transformation exists if $K_2$ ``sticks out'' of $K_1$ in at least two places. One could formalize and generalize this argument by using the degree of the map from the boundary of the domain of $F$ to the $(d-1)$-dimensional sphere.

\begin{prb}
Give a sufficient condition for the existence of a projective transformation in the case when neither $K_1 \subset \conv K_2$ nor $K_2 \subset \conv K_1$.
\end{prb}

\end{document}